\documentclass[12pt,leqno]{amsart}
\usepackage[total={6in,9in},
top=1in, left=1in, right=1in, bottom=1in]{geometry}
\usepackage{amsmath}
\usepackage[italian, english]{babel}
\usepackage{amsfonts}
\usepackage{amssymb}
\usepackage{dsfont}
\usepackage{mathrsfs}
\usepackage{graphicx}
\usepackage{float}
\usepackage{fancyhdr}
\usepackage{multirow}
\usepackage{multicol}
\usepackage{enumitem}
\usepackage{amsthm}
\usepackage{empheq}
\usepackage{cases}
\usepackage[all]{xy}
\usepackage{stmaryrd}
\usepackage[colorlinks, citecolor=blue, linkcolor=red, pagebackref=true]{hyperref}
\usepackage{mathrsfs}
\usepackage{tikz,tikz-cd}
\usepackage[font=small,labelfont=bf]{caption}
\usepackage[pagewise]{lineno}

\makeatletter
\def\l@subsection{\@tocline{2}{0pt}{2.5pc}{4pc}{}}
\makeatother

\makeatletter
\newcommand{\tocless}[2]{\begingroup\let\addcontentsline\@gobblethree#1{#2}\endgroup}
\makeatother

\def\SS{{{\mathbb S}}}
\def\NN{{{\mathbb N}}}

\def\RR{{\mathbb R}}

\tikzset{
	subset/.style={
		draw=none,
		edge node={node [sloped, allow upside down, auto=false]{$\subset$}}},
	Subset/.style={
		draw=none,
		every to/.append style={
			edge node={node [sloped, allow upside down, auto=false]{$\subset$}}}
	}
}
\tikzset{
	labl/.style={anchor=south, rotate=90, inner sep=.50mm}
}

\newcommand{\erre}{\mathds{R}}

\newcommand{\cinf}{C^{\infty}(M)}

\newcommand{\del}[1]{\delta_{#1}}

\newcommand{\riem}{\operatorname{Riem}}
\newcommand{\ricc}{\operatorname{Ric}}
\newcommand{\weyl}{\operatorname{W}}
\newcommand{\bach}{\operatorname{B}}
\newcommand{\diver}{\operatorname{div}}
\newcommand{\hess}{\operatorname{Hess}}

\newcommand\restrict[1]{\raisebox{-.5ex}{$|$}_{#1}}

\newcommand{\KN}{\mathbin{\bigcirc\mspace{-15mu}\wedge\mspace{3mu}}}

\newcommand{\longra}{\longrightarrow}

\newcommand{\Longra}{\Longrightarrow}

\newcommand{\set}[1]{{\left\{#1\right\}}}               
\newcommand{\pa}[1]{{\left(#1\right)}}                  
\newcommand{\sq}[1]{{\left[#1\right]}}                  
\newcommand{\abs}[1]{{\left|#1\right|}}                 
\newcommand{\norm}[1]{{\left\|#1\right\|}}              

\newcommand{\eps}{\varepsilon}                           

\newcommand{\ol}[1]{\overline{#1}}

\renewcommand{\tilde}[1]{\widetilde{#1}}

\newcommand{\pmat}[1]{{\begin{pmatrix}#1\end{pmatrix}}} 

\newtheorem{theorem}{\textbf{Theorem}}[section]
\newtheorem{lemma}[theorem]{\textbf{Lemma}}
\newtheorem{proposition}[theorem]{\textbf{Proposition}}
\newtheorem{cor}[theorem]{\textbf{Corollary}}

\theoremstyle{remark}
\newtheorem{rem}[theorem]{\textbf{Remark}}

\numberwithin{equation}{section}
\title[On the Bach tensor and quadratic curvature functionals]
{
	On the Bach tensor and quadratic curvature functionals
}

\date{\today}


\begin{document}
	\maketitle
	
	\date{\today}
	
	\begin{center}
		\textsc{\textmd{Letizia Branca\footnote{Università degli Studi di Milano, Italy. Email: letizia.branca@unimi.it} and
				Davide Dameno \footnote{Universit\`{a} degli Studi di Torino, Italy.
					Email: davide.dameno@unito.it.}. }}
	\end{center}

	\begin{abstract}
		We study the critical points of a quadratic functional depending on the gradient of the Bach tensor on Riemannian four-manifolds, which generalize the Bach-flat condition. We show that, on every closed four-manifold, there exists a weak Bach-parallel metric, i.e. a critical point for this functional with respect to conformal variations: in particular, we prove that there exist infinitely many conformal classes which contain a unique minimizer for the functional, up to constant positive rescaling. Next, we analyze the global minima of the functional, i.e. metrics with parallel Bach tensor, relating these metrics to well-known variational problems. Using a version of de Rham's splitting theorem on complete four-manifolds, we provide a classification result for products of surfaces, exploiting the theory of conformal gradient solitons; we also construct a new explicit example of a Bach-flat metric which is neither locally conformally flat nor conformally Einstein and we characterize HCMU metrics on complete surfaces. Finally, we prove an equivalence between the Bach-parallel condition on 4D cylinders and the existence of critical metrics for a well-known quadratic curvature functional in dimension three: in this direction, we also prove a characterization of flat three-manifolds, under some curvature and finite energy assumptions. 
	\end{abstract}
	
	\tableofcontents

	\section{Introduction and main results}
	Let $(M,g)$ be a closed, Riemannian four-manifold: we can define the so-called \emph{Weyl functional} as the squared $L^2$-norm of the \emph{Weyl tensor} $\weyl_g$ (the trace-free, conformally invariant part of the Riemann curvature $\riem$), i.e. 
	\[
	\mathfrak{W}(g):=\int_M\abs{\weyl_g}_g^2dV_g.
	\]
	It is well known that this functional is 
	conformally invariant and its critical points are determined by the Euler--Lagrange equation \cite{besse}
	\[
	\bach_g=0,
	\] 
	where $\bach_g$ is the so-called \emph{Bach tensor}
	\cite{Bach1921}: this quantity is a trace-free, symmetric $(0,2)$-tensor which, due to the conformal properties of $\mathfrak{W}$, also happens to be conformally covariant in dimension four, i.e. 
	$e^{2u}\bach_{\tilde{g}}=\bach_g$, if 
	$\tilde{g}=e^{2u}g$, where $u\in\cinf$. We say that a Riemannian metric $g$ is \emph{Bach-flat} if $\bach_g=0$ on $M$, i.e., if it is a critical point of $\mathfrak{W}$: we remark that the Bach tensor $\bach_g$ (and, consequently, Bach-flat metrics) can be defined in every dimension $n\geq 4$, although
	$\bach_g$ loses its conformal covariance property and Bach-flat metrics are not critical points of $\norm{\weyl}_{L^{n/2}(M)}^{n/2}$ if $n>4$.\\
	Examples of Bach-flat metrics in dimension four are provided by \emph{conformally Einstein metrics} and \emph{half conformally flat metrics} \cite{besse}; in the literature, there are very few examples which do not belong to these categories (sometimes called \emph{strictly Bach-flat}), for instance the ones constructed in \cite{abbena} (see also \cite{garciariobach}) and in \cite{garciariosoliton, leistner}. 
	Up to now, no topological obstruction is known on four-dimensional manifolds for the existence of Bach-flat metrics; however, some rigidity and classification results have been provided under some additional assumptions, e.g. for gradient Ricci solitons, in the K\"ahler setting or under suitable curvature conditions (see, for instance, \cite{caochencat, caochen, changgurskyyang, lebrunbach, tianviac2, tianviac, zhang}). On the other hand, in \cite{gurskyviac}, Gursky and Viaclovsky managed to construct examples of critical metrics arising as critical points of the Weyl functional perturbed with a quadratic scalar curvature term, known as $B^t$-flat metrics. 
	In the opposite direction, some results concerning the existence of Riemannian metrics whose Bach tensor satisfies suitable properties have been proven recently: 
	indeed, in \cite{cdmbach}, the second author, together with Catino and Mastrolia, exploited a deformation argument due to Aubin \cite{Aubin} to show that every compact, smooth four-manifold admits a Riemannian metric with nowhere vanishing Bach tensor, while in \cite{smallbach} the authors, together with Catino, proved that, on every closed four-dimensional manifold, there
	exists a metric whose Bach tensor is $\eps$-pinched by the scalar curvature.
	
	The purpose of this paper is to analyze a new curvature condition, which generalize the notion of Bach flatness, and to prove that metrics which satisfy this property are not obstructed. Namely, given a Riemannian four-manifold $(M,g)$, we introduce 
	a scaling invariant functional only depending on 
	the Bach tensor, defined as 
	\[
	\mathcal{B}(g):=\mathrm{Vol}_g(M)^{\frac{3}{2}}\int_M\abs{\nabla\bach_g}_g^2dV_g,
	\]
	and we ask whether a critical point for $\mathcal{B}$ always
	exists on every closed four-manifold.
	This
	definition is inspired by the functional considered by 
	Catino, Mastrolia, Monticelli and Punzo \cite{cmmp}, i.e. the normalized $L^2$-norm of the divergence of $\weyl_g$: we point out that we cannot completely adapt their definition to our setting, since, in dimension four, the Bach tensor is always divergence-free, hence we take into account its full covariant derivative. Following \cite{cmmp}, 
	we want to consider conformal variations of the functional $\mathcal{B}$ and to show that there always exists a conformal class of metrics $[g]$ which contains a minimizer for $\mathcal{B}$, on every closed four-manifold. We say that a metric is \emph{weak Bach-parallel} if it is a critical point of $\mathcal{B}$ for such variations, mirroring the definition of weak harmonic Weyl metrics given in \cite{cmmp}: it is clear that all Bach-flat metrics (or, more generally, \emph{Bach-parallel metrics}, i.e. the ones with 
	$\nabla\bach\equiv 0$) are minima for full variations of $\mathcal{B}$, which, in particular, means that they are weak Bach-parallel. Although the possible topological obstructions to Bach flatness are not known yet, we manage to prove an existence result for weak Bach-parallel metrics:
	
	\begin{theorem} \label{weakbachparallel}
		Let $M$ be a closed, smooth four-manifold: then, there exists a weak Bach-parallel metric $g$ on $M$. More precisely, there exist infinitely many conformal classes $[\ol{g}]$ which contain a minimizer of $\mathcal{B}$ with respect to conformal variations. Moreover, the minimizer is unique, up to multiplication by a positive constant.
	\end{theorem}
	Our approach is based on the strategy adopted in \cite{cmmp} and on the aforementioned existence result proven in \cite{cdmbach}. More precisely, the Euler--Lagrange equation for critical points of $\mathcal{B}$ with respect to conformal variations can be viewed as a second-order elliptic PDE in the conformal factor $v$
	\[
	-\diver_g\pa{G\nabla v}+cv=\dfrac{\lambda(v)}{v^{\frac{7}{3}}},
	\]
	(see \eqref{pdetosolve} for a precise description of the PDE):
	if we restrict our attention to a conformal class of metrics $[g]$ whose Bach tensor nowhere vanishes \cite{cdmbach}, we can show that the operator appearing on the left-hand side of the PDE becomes uniformly elliptic. Following \cite{cmmp}, we then construct a function $v$, which attains the infimum of $\mathcal{B}$, inside a suitable subset of $H^1(M)$ and, making use of a version of the maximum principle, we are able to show that $v$ is a weak solution of the PDE, avoiding the possible issues caused by the singularity given by the term $v^{-7/3}$. Then, a standard bootstrapping argument shows that the solution $v$ is smooth and positive, which means that it can be regarded as a classical solution of the PDE: hence, $v$ is a true conformal factor and the metric $\tilde{g}=v^{-2/3}g$ is a smooth minimizer for $\mathcal{B}$. Furthermore, $\abs{\bach_g}_g>0$ on $M$ is an open condition, which means that, if $\tilde{g}$ is a small deformation of $g$ inside the space of $C^4$ metrics, then $\abs{\bach_{\tilde{g}}}_{\tilde{g}}>0$ on $M$, i.e. in every such conformal class there exists a minimizer for $\mathcal{B}$ (this also comes from the ``genericity property" that the metrics constructed in \cite{Aubin, cdmbach} satisfy). We highlight the fact that, in general, finding minimizers with respect to conformal variations for such a functional is a degenerate elliptic problem and, therefore, the method we propose cannot be generally applied. 
	
	Next, we provide some results concerning the absolute minima of the functional $\mathcal{B}$, which are Bach-parallel metrics. In particular, we focus on Riemannian metrics $g$ such that $\nabla\bach_g\equiv 0$ but $\bach_g\not\equiv 0$, in order to derive some classification results and to provide non-trivial examples of Bach-parallel metrics. The key tool we use is a version of the well-known de Rham's splitting Theorem \cite{boubel, eisen}: if $g$ is complete and $\bach_g$ is parallel, then either $\bach_g$ is identically zero or the manifold is (locally) isometric to a Riemannian product (see Theorem \ref{splittingbach}). This essentially reduces our analysis to the study of Bach-parallel, non-Bach-flat metrics on products of surfaces or on manifolds of the form $\mathbb{S}^1\times N^3$, where $N^3$ is a three-dimensional manifold. In the first case, using local computations on the Bach tensor of a product metric on $N_1\times N_2$, where $N_i$ is a surface, and some standard ODE analysis, we obtain a full classification of such metrics:
	\begin{theorem} \label{bachpar22}
		Let $(M,g)$ be a simply connected, complete, 
		four-dimensional Riemannian manifold with 
		$\nabla \bach\equiv 0$ and $\bach\not\equiv 0$; assume that
		$(M,g)$ is isometric to a Riemannian product $(N_1(K_1)\times N_2(K_2), g_1\oplus g_2)$, where $N_i(K_i)$ is a Riemannian surface with Gaussian curvature $K_i$. Then, either $K_i$ is constant for $i=1,2$, with $K_1\neq - K_2$, or, for some $i$, $N_i=\RR^2$ and
		\begin{equation} \label{rotsymmmetrics}
			g_i=dr^2+(h(r))^2d\theta^2,
		\end{equation}
		where 
		\[
		h(r)=\dfrac{2}{\sqrt{-a}}\tanh\pa{\dfrac{r\sqrt{-a}}{2}}\operatorname{sech}^2\pa{\dfrac{r\sqrt{-a}}{2}}, \quad a<0,
		\]
		and $(N_i,g_i)$ has a hyperbolic, cusp-like end. 
	\end{theorem}
	One interesting consequence of Theorem \ref{bachpar22} is the fact that there exists a ``non-trivial" example of a Bach-parallel metric on a product of surfaces: more precisely, we show that, when the Gaussian curvature $K_i$ of $(N_i,g_i)$ is not constant, this surface becomes a \emph{conformal gradient soliton} with potential function $K_i$, which means that
	\[
	\hess K_i=\dfrac{\Delta K_i}{2}g_i.
	\]
	This equation, with a general, non-constant potential $f$, has been extensively studied over the years and complete Riemannian manifolds which satisfy this condition are fully classified
	(see e.g. \cite{catmantmazz,osgood,tashiro}): in our case, since $f=K_i$, we obtain that $K_i$ must be a solution of a static KdV equation and, exploiting results in \cite{ishiharatashiro, tashiro}, we are able to show that $K_i$ has exactly one critical point and that $g_i$ is rotationally symmetric (the explicit expression of $h$ is derived by solving another ODE, once we explicitly obtain $K_i$).
	
	Theorem \ref{bachpar22} has two other interesting consequences: 
	first, if we consider the complete Riemannian four-manifold
	$(M=N_1\times N_2, g=g_1\oplus g_2)$, where $g_i$ is a metric of the one-parameter family \eqref{rotsymmmetrics} with
	parameter $a_i$, there exists a choice of the $a_i$'s such that
	$g$ is Bach-flat (namely, $a_1=a_2$). In this case, it can 
	be shown that $g$ is neither locally conformally flat nor 
	conformally Einstein, which means that, in particular, $g$ is
	strictly Bach-flat: this new example is explicit, non-homogeneous, rotationally symmetric with sign-changing scalar curvature.  
	\begin{theorem}\label{thm example}
		Let $(M,g)$ be a complete Riemannian four-manifold and assume that it is isometric to the Riemannian product of two surfaces $(N_1(K_1)\times N_2(K_2),g_1\oplus g_2)$, such that
		\[g_i=dr_i^2+(h_i(r_i))^2d\theta_i^2,\quad i=1,2\]
		where 
		\[ h_i(r_i)=\dfrac{2}{\sqrt{-a}}\tanh\pa{\dfrac{r_i\sqrt{-a}}{2}}\operatorname{sech}^2\pa{\dfrac{r_i\sqrt{-a}}{2}}, \quad a<0, \quad i=1,2.\]
		Then, $(M,g)$ is strictly Bach-flat.
	\end{theorem} 
	If, on one hand, it is immediate to verify that $g$ is not locally conformally flat, on the other hand, in order to show that $[g]$ does not admit an Einstein representative we argue by contradiction, showing that the only possible conformal factor is well-defined on $M\setminus Z$, where $Z$ is the zero set of the sum of the Gaussian curvatures: this provides a conformally Einstein metric on $M\setminus Z$ (that is unique, up to constant rescaling \cite{derdzinski}), which cannot be smoothly extended to $Z$. 
	
	The other implication derived from Theorem \ref{bachpar22}
	regards the so-called \emph{HCMU} metrics (``Hessian of the curvature of the metric is umbilical", see e.g. \cite{chenwuxu, chenextremal, chenobstruction}), i.e. Riemannian metrics on surfaces whose Gaussian curvature has vanishing trace-free Hessian (which is equivalent to say that the surface admits a conformal gradient soliton structure with potential $K$): these metrics are the critical points of the well-known 
	\emph{Calabi functional} \cite{calabi1, calabi2}
	\[
	\mathfrak{E}(g):=\int_M K_g^2 dV_g
	\]
	with respect to (compactly supported) full variations that
	fix the (relative) volume and, in the compact case, 
	are best known as \emph{extremal K\"{a}hler metrics}
	\cite{szeke2} (see Subsection \ref{secction extHerm} for a more detailed dissertation). In the closed case, the HCMU condition 
	is equivalent to the equation 
	\begin{equation} \label{ELcalabi}
		\Delta K+K^2=C,
	\end{equation}
	which is the Euler--Lagrange equation of $\mathfrak{E}$ with respect to conformal variations that preserve the volume: 
	however, if the surface is only complete, \eqref{ELcalabi} does not, in general, imply the vanishing of $\mathring{\hess}(K):=\hess(K)-(\Delta K)/2\cdot g$. Metrics which satisfy
	\eqref{ELcalabi} are called \emph{extremal Hermitian}:
	in much of the literature, these metrics are characterized and classified under finite energy and finite area assumptions
	(see e.g. \cite{chenwuexplicit, chenwuxu, chenextremal, wangzhu}) and they can be seen as generalization
	of extremal K\"{a}hler metrics on punctured Riemann surfaces. 
	Theorem \ref{bachpar22} therefore provides a classification
	result for HCMU metrics without any additional hypothesis and
	for extremal Hermitian metrics, under an integrability hypothesis on $\nabla K$. 
	\begin{cor} \label{corextKahler}
		Let $(M,g)$ be a complete Riemannian surface. Then, 
		\begin{itemize}
			\item if $g$ is a HCMU metric, then either $K$ is
			constant or
			\[
			g=dr^2+(h(r))^2d\theta^2,
			\]
			where $h$ belongs to the family given in the statement of Theorem \ref{bachpar22};
			\item if $g$ is an extremal Hermitian metric, we have the same conclusion, provided that $\nabla K\in L^2(M)$.
		\end{itemize}
	\end{cor}
	
	On the other hand, when $(M,g)$ splits as 
	a product of the form $\mathbb{S}^1\times N$, where $(N,g_N)$ is a Riemannian three-manifold, we cannot establish a full classification of Bach-parallel metrics: indeed, the local components of the Bach tensor of such metric are much more involved and rather difficult to deal with in a satisfying manner. However, it is possible to observe that, when $M$ splits this way, the metric $g$ is Bach-parallel if and only if $g_N$ is critical for a suitable quadratic curvature functional: namely, we have the following
	\begin{theorem} \label{criticalequivalenceintro}
		Let $(N,g_N)$ be a complete Riemannian three-manifold. Then, the product metric $g=g_{\SS^1}\oplus g_N$ defined on $M=\mathbb{S}^1\times N$ is Bach-parallel if and only if $g_N$ is critical for the Riemannian functional 
		\[
		\mathfrak{F}_{-1/3}(h)=\int_N\abs{{}^N\ricc_h}_h^2 \, dV_h -\dfrac{1}{3}\int_N{}^NS_h^2dV_h. 
		\]
	\end{theorem}
	More in detail, if we denote by 
	$\nabla\mathfrak{F}_{-1/3}$ the gradient of $\mathfrak{F}_{-1/3}$, $g_N$ is a critical point if and only if there exists $c\in\RR$ such that $\nabla\mathfrak{F}_{-1/3}=c\cdot g_N$: this result, together with Theorem \ref{splittingbach}, shows that $g_N$ is critical if and only if $\bach_g$ is parallel with
	\[
	\bach_g=\pmat{
		-3c/2 & & &\\
		& c/2 & & &\\
		& & c/2 & &\\
		& & & c/2
	}.
	\]
	In particular, $g$ is Bach-flat if and only if 
	$c=0$, i.e. $g_N$ is critical with respect to full variations, in the closed case, or to compactly supported variations in
	the complete, non-compact case (for a more detailed discussion, see Section \ref{secction global min II}): Theorem \ref{criticalequivalenceintro} therefore suggests that one can construct new Bach-flat metrics on complete, non-compact four manifolds of the form $\SS^1\times N$ by considering
	critical points of $\mathfrak{F}_{-1/3}$ with respect to compactly supported full variations on the three-dimensional factor (we recall that, if $(N,g_N)$ is closed and critical for $\mathfrak{F}_{-1/3}$ with respect to full variations, then $g_N$ has constant sectional curvature).
	
	The study of quadratic curvature functionals, started in \cite{bergerRiemfunc}, has been of interest both in mathematics and physics, as these functionals are related to actions of certain gravitational field theories 
	(see \cite[Chapter 4]{besse} for a more detailed discussion). Over the years, several authors have studied quadratic curvature functionals, not depending on $\mathfrak{W}$, on smooth $n$-dimensional Riemannian manifolds, namely functionals of the form
	\[\mathfrak{F}_{t}(g)=\int_M\abs{\ricc_g}_g^2 \, dV_g +t\int_MS^2_gdV_g,\]
	for $t\in \erre\cup \set{+\infty}$, where the case $t=+\infty$ formally corresponds to the $L^2$-norm of the scalar curvature of $g$. These functionals are scaling invariant only in dimension four, therefore, when $M$ is closed, it is natural to restrict the functional to $\mathscr{M}_1(M)$, i.e. the class of metrics with $\mathrm{Vol}(M)=1$. 
	It is well known that Einstein metrics are critical points for $\mathfrak{F}_{t}$ on $\mathscr{M}_1(M)$, but, in general, they are not the only ones: for instance, it is important to remark that, on closed four-manifolds, by the Chern--Gauss--Bonnet formula $\mathfrak{F}_{-1/3}$ coincides with the Weyl functional, up to adding a multiple of the Euler characteristic of $M$, therefore, they have the same critical points, i.e. are Bach-flat metrics, which, in general, are not Einstein. Moreover, some non-Einstein homogeneous examples have been constructed for $4\mathfrak{F}_{-1/4}$ in \cite{lamontagne} on $\mathscr{M}_1(\mathbb{S}^3)$ and, in \cite{garciario3homog}, the existence of non-Einstein homogeneous critical metrics for quadratic curvature functional on three-manifolds is proven. As a consequence, when $(M^4,g)$ splits as $\mathbb{S}^1\times N$, the examples in \cite{garciario3homog} for $t=-1/3$ lead to the existence of non-trivial Bach-parallel metrics which are not Bach-flat on $M$. This imply that the Bach-parallel condition on a product of the form $\mathbb{S}^1\times N$ does not immediately imply that $N$ has constant sectional curvatures, which is instead implied by the Bach-flat condition on the product metric in the closed case (see e.g. \eqref{bach13comp}).
	On the other hand, it has been proven that  
	a complete non-compact three-manifold which is critical for $\mathfrak{F}_{t}$, $-1/3<t<+\infty$, has to be flat \cite{catmastrmont}; the same conclusion holds if $t=-1/3$, provided that the scalar curvature is non-negative \cite{anderson2}.
	Note that, more in general, some rigidity results and examples for critical points of $\mathfrak{F}_t$  have been proven in \cite{bernardini,catinocalcvar,catmastrmont,cmm2026,gurvia3d,gurviacritical,gurskyviac}. To the best of our knowledge, there exist no results in the literature concerning critical metrics 
	for $\mathfrak{F}_{-1/3}$ in the non-compact case with respect to compactly supported variations which preserve the relative volume, i.e. the ones restricted to symmetric tensors $h$ such that 
	\[
	\int_M\operatorname{tr}(h)=0;
	\]
	for $t=-1/3$, we are able to provide the following rigidity theorem:
	\begin{theorem} \label{thmFintro}
		Let $(N,g_N)$ be a complete Riemannian manifold of dimension $3$ such that $g_N$ is a critical metric for
		$\mathfrak{F}_{-1/3}$ and let $c\geq 0$ such that 
		$(\nabla\mathfrak{F}_{-1/3})=-c\cdot g_N$.
		Assume that 
		\begin{itemize}
			\item $\abs{\ricc}^2\leq 6c$ on $N$;
			\item there exists $p_0\in N$ such that $\mathrm{Vol}(B_R(p_0))\leq C\cdot R^4\ln R$, for 
			$R>>1$;
		\end{itemize}
		then $S\equiv 0$. If, moreover, 
		$\nabla\ricc\in L^2(N)$, then, after passing at worst to a double cover of $N$, $g_N$ is flat.
	\end{theorem}
	We point out that in the first part of the proof of Theorem \ref{thmFintro} we rely on a result of Grigor'yan and Sun \cite{grigsun} and on the Omori--Yau maximum principle to conclude that the scalar curvature of $(N,g_N)$ is zero: the finite energy assumption on
	$\nabla\ricc$, together with some algebraic estimates, allows us to deduce that either $(N,g_N)$ is flat or it has finite volume.
	At this point, we note that $\ricc$ can be seen 
	as a linear, self-adjoint, trace-free operator 
	on $\Lambda^1N$, which means that, to some extent, 
	it shares some algebraic properties with the \linebreak
	(anti-)self-dual Weyl operator 
	$\weyl^{\pm}$ in dimension four: hence, using the resulting parabolicity of $(N,g_N)$, we are able to
	adapt arguments used by Catino and the second author \cite{catinodameno}, inspired by the work of LeBrun \cite{lebrun}, to find the desired contradiction in the finite volume case. We remark that, with this technique, the flatness conclusion holds at worst on a double cover of $N$: indeed, for our computations to hold, it is necessary to assume that the line bundle generated by the eigenspace of a simple eigenvalue of $\ricc$ is trivial, which is always true, up to passing to a double cover. 
	We note that the assumption on the volume growth of $(N,g_N)$ is used to show that there are no non-negative solutions to the inequality $\Delta u+u^2\leq 0$ \cite{grigsun}: if we remove or weaken this hypothesis, we might not be able to conclude that $S\equiv0$ (see Remark \ref{remvolgr} for more details).
	
	Finally, we obtain a rigidity result under an extra assumption on the Cotton tensor of the four-manifold: indeed, we show that, if $(M,g)$ is a complete four-manifold with parallel Bach tensor and vanishing Cotton tensor, then either $(M,g)$ is Bach-flat or it splits as a product of surfaces with constant sectional curvatures 
	(Theorem \ref{bach3cott}). We highlight the fact that the proof of this result relies on local computations inspired by 
	the work of G\l{}odek \cite{glodek}: as shown later, this implies that the result can be generalized, to some extent, to the case where there exists a general symmetric, parallel $(0,2)$-tensor on $(M,g)$ (see Remark \ref{remglodek}).
	
	This paper is organized as follows. In Section \ref{section prel}, we fix the notation and we recall some well-known preliminary facts; in Section \ref{functionalsection}, we introduce the functional $\mathcal{B}$ and its critical points, computing its Euler--Lagrange equations with respect to conformal variations and showing how it can be regarded as a PDE for the conformal factor. Section \ref{section proof main thm} is devoted to the proof of Theorem \ref{weakbachparallel}. Section \ref{secction global min I} is dedicated to global minima of $\mathcal{B}$ on products of surfaces, for which we prove Theorem \ref{bachpar22}, Theorem \ref{thm example} and Corollary \ref{corextKahler}; Section \ref{secction global min II} concerns the Bach-parallel condition on products of the form $\SS^1\times N^3$ and contains the proofs of Theorem \ref{criticalequivalenceintro}, Theorem \ref{thmFintro} and
	Theorem \ref{bach3cott}. 
	\addtocontents{toc}{\protect\setcounter{tocdepth}{0}}
	\section*{Acknowledgements}
	
	We would like to sincerely thank Dario Monticelli for helpful and valuable discussions concerning Section \ref{section proof main thm}. We also express our deep gratitude to Giovanni Catino, Alberto Farina, Michele Gatti, Lorenzo Maniscalco and Luciano Mari for their interest in our work and for their insightful inputs regarding the making of Section \ref{secction global min I} and Section \ref{secction global min II}. We also thank Paolo Mastrolia for reading the first draft of the work and giving us useful suggestions which highly improved the paper. Both authors are members of the Gruppo Nazionale per le Strutture Algebriche, Geometriche e loro Applicazioni (GNSAGA) of INdAM (Istituto Nazionale di Alta Matematica). 
	
	\vspace{0.2cm}
	\textbf{AI usage statement.} No AI tool was used to generate ideas or proofs of the theorems and the text of this paper was completely written by the authors. AI tools were only used to find bibliographic references and to check some technical details in Step 2 of the proof of Theorem \ref{thmFintro}.
	\addtocontents{toc}{\protect\setcounter{tocdepth}{3}}
	
	\section{Preliminaries and notation}\label{section prel}
	Given a $n$-dimensional Riemannian manifold $(M,g)$, we adopt the following convention for the Riemann tensor: the $(1,3)$-version is given by
	\[
	R(X,Y)Z=\nabla_X\nabla_YZ-\nabla_Y\nabla_XZ-\nabla_{[X,Y]}Z, \qquad \forall\, X,\,Y,\,Z \in \mathfrak{X}(M)
	\]
	while we define its $(0,4)$-version as
	\[
	\mathrm{Riem}_g(X,Y,Z,W)=g(R(Z,W)Y,X)\,, \qquad \forall\: X,\,Y,\,Z,\,W \in \mathfrak{X}(M).
	\]
	Throughout the paper, we will denote the local components
	of the covariant derivative of tensors as follows: 
	given a local orthonormal frame $\{e_1,...,e_n\}$ on an open chart of $M$ and a $(0,r)$-tensor field $\operatorname{T}$, 
	the components of $\nabla\operatorname{T}$ with respect
	to the chosen frame will be written as $T_{i_1,...,i_r,k}$,
	where 
	\begin{equation} \label{covderformula}
		T_{i_1,...,i_r,k}=e_k(T_{i_1,...,i_r})-
		\sum_{j=1}^rT_{i_1,...,i_{j-1}pi_{j+1},...,i_r}\omega_{i_j}^p(e_k)
	\end{equation}
	and $\{\omega_j^i\}_{i,j=1}^n$ are the Levi--Civita connection forms associated to the dual coframe 
	$\{\omega^1,...,\omega^n\}$ \emph{via} Cartan's first structure equations (see e.g. \cite{cmbook}). We point out
	that, unless specified otherwise, we adopt Einstein's summation convention over repeated indices; we also
	write $T_{i_1,...,i_r,k_1,...,k_s}$ for the components of 
	$\nabla_{k_s}\nabla_{k_{s-1}}\cdots\nabla_{k_1}\operatorname{T}$, where 
	$\nabla_k=\nabla_{e_k}$. We will always use the convention $\Delta=\operatorname{tr}_g(\hess)$ both for 
	smooth functions and tensors: with respect to a local orthonormal frame, we write 
	\[
	\Delta f=f_{kk}, \quad \Delta\operatorname{T}_{i_1,...,i_r}=T_{i_1,...,i_r,kk}.
	\]
	
	It is well known that, when $n\geq3$, the Riemann curvature tensor can be written as 
	\begin{align}\label{riemdec}
		\operatorname{Riem}_g=\weyl_g+\dfrac{1}{n-2}
		\ricc_g\KN g
		-\dfrac{S_g}{2(n-1)(n-2)}g\KN g,
	\end{align}
	where $\weyl_g$, $\ricc_g$ and $S_g$ denote the \emph{Weyl tensor}, the \emph{Ricci tensor} and the \emph{scalar curvature} of $g$, respectively, and $\KN$ is the Kulkarni--Nomizu product. With respect to a (local) orthonormal frame equation \eqref{riemdec} rewrites as
	\[R_{ijkt}=W_{ijkt}+\dfrac{1}{n-2}\pa{R_{ik}\delta_{jt}
		-R_{it}\delta_{jk}+R_{jt}\delta_{ik}-R_{jk}\delta_{it}}
	-\dfrac{S}{(n-1)(n-2)}\pa{\delta_{ik}\delta_{jt}-\delta_{it}\delta_{jk}},\]
	where $R_{ij}=R_{ikjk}$ and $S=R_{ii}$.
	
	The Weyl tensor has the same symmetries as the Riemann curvature tensor, it is trace-free and for $n=3$ vanishes identically; most importantly, the $(1,3)$-version is conformally invariant, i.e., if we define
	$\tilde{g}=e^{2u}g$, 
	where $u\in\cinf$, then 
	\[
	\weyl_{\tilde{g}}=\weyl_g.
	\] 
	Throughout this paper, given a metric $\tilde{g}$ in the conformal class $[g]$, we might write 
	$\tilde{\weyl}$ instead of $\weyl_{\tilde{g}}$, for
	the sake of simplicity: this will be true also
	for the other geometric quantities. 
	
	The Weyl tensor is obviously strictly related to the conformal structure of $(M,g)$: 
	indeed,  when $n\geq4$, $g$ is \emph{locally conformally flat} (i.e. it can be locally rescaled in order to obtain a flat metric), if and only if $\weyl\equiv 0$. Since in dimension three $\weyl$ identically vanishes for every Riemannian metric, the local conformal flatness is equivalent to the vanishing of the so-called \emph{Cotton tensor} 
	$\operatorname{C}$, which is locally defined as
	\[
	C_{ijk}=R_{ij,k}-R_{ik,j}-\frac{S_k}{2(n-1)}\delta_{ij}+\frac{S_j}{2(n-1)}\delta_{ik}.
	\]
	The Cotton tensor is skew-symmetric with respect to the last two indices, it is totally trace-free and it satisfies the following identities
	\[
	C_{ijk}+C_{jki}+C_{kij}=0,\qquad C_{ijk,i}=0;
	\] 
	furthermore, a computation shows that, for $n\geq4$, the Cotton tensor is proportional to the divergence of the Weyl tensor
	\[C_{ijk}=\pa{\frac{n-2}{n-3}}W_{tikj,t}\]
	(see e.g. \cite{cmbook}).\\
	The \emph{Schouten tensor} $\operatorname{P}$ is defined as
	\[\operatorname{P}=\ricc-\frac{S}{2(n-1)}g;\]
	it is a symmetric $(0,2)$-tensor and its components are
	\[P_{ij}=R_{ij}-\frac{S}{2(n-2)}\delta_{ij}.\]
	Using the Schouten tensor we rewrite the decomposition \eqref{riemdec} as
	\[\riem=\weyl+\frac{1}{n-2}\operatorname{P}\KN g;\]
	we point out that the Cotton tensor can be seen as the obstruction to $\operatorname{P}$ being Codazzi, since
	\[C_{ijk}=P_{ij,k}-P_{ik,j}.\]
	The curvature quantity we are interested in, which is deeply related to conformal geometry, is the \emph{Bach tensor}: when $n\geq4$, it is locally defined as
	\begin{align}\label{bachdef}
		B_{ij}=\dfrac{1}{n-3}W_{ikjl,lk}+\dfrac{1}{n-2}R_{kl}W_{ikjl}=
		\dfrac{1}{n-2}\pa{C_{ijk,k}+R_{kl}W_{ikjl}},
	\end{align}
	while for $n=3$ its components are
	\[B_{ij}=C_{ijk,k}.\]
	The Bach tensor is trace-free and symmetric; moreover it can be expressed using the following formula
	\begin{align}\label{bachdef2}
		(n-2)B_{ij}=&\Delta R_{ij}-\frac{1}{2} \cdot \left(\frac{n-2}{n-1}\right)S_{ij}-\left(\frac{n-4}{n-2}\right)R_{ik}R_{kj}+2R_{kl}R_{ikjl}\\
		&-\frac{n}{(n-1)(n-2)}SR_{ij}-\frac{1}{2(n-1)}\Delta S\delta_{ij}-\frac{1}{n-2}\Big(\vert \ricc \vert^2-\frac{S^2}{n-1}\Big)\delta_{ij}\notag
	\end{align}
	(see e.g. \cite[Lemma 1.5]{cmbook} for a proof). This geometric quantity was first introduced by Bach \cite{Bach1921} in the context of Conformal Relativity and in dimension four it is divergence-free and \emph{conformally covariant}, i.e. given $\tilde{g}=e^{2u}g$,
	\[
	\tilde{\bach}=e^{-2u}\bach,
	\]
	where $\tilde{\bach}$ is the Bach tensor of $\tilde{g}$: as a consequence, the condition 
	\begin{align}\label{bachflat}
		\bach\equiv 0
	\end{align}
	is conformally invariant. Riemannian metrics satisfying \eqref{bachflat} are said to be \emph{Bach-flat}: as we recalled in the Introduction, these metrics arise as critical points of the so-called \emph{Weyl functional} in dimension four, i.e. 
	\[
	\mathfrak{W}(g)=\int_M\abs{\weyl}^2dV_g.
	\]
	This functional can be defined in every dimension as $\norm{\weyl}_{L^{n/2}(M)}^{n/2}$, although the critical points are not characterized by the Bach-flat condition if $n>4$.
	
	
	While relations between topology and the existence of Bach-flat metrics in dimension four are not well understood, if $n=3$, Bach-flatness is equivalent to local conformal flatness:
	\begin{proposition} \label{bachflat3}
		Let $(M,g)$ be a three-dimensional Riemannian manifold. Then, $g$ is Bach-flat if and only if it is locally conformally flat.
	\end{proposition}
	\begin{proof}
		Since, in dimension $3$, 
		\[
		B_{ij}=C_{ijk,k},
		\]
		it is clear that every locally conformally flat 
		metric is also Bach-flat.
		Now, assume that $g$ is Bach-flat:
		we recall that, for every Riemannian manifold $(M,g)$ of dimension $n\geq 3$,
		\[
		B_{ij,j}=\dfrac{n-4}{(n-2)^2}C_{jki}R_{jk}:
		\]
		(see e.g. \cite{caochencat} for 
		a proof in the three-dimensional case and \cite{cmbook} for a general proof).
		
		In particular, since $g$ is Bach-flat by hypothesis, we have that
		\begin{equation} \label{cottonid}
			C_{ijk}R_{ij}=0. 
		\end{equation}
		By differentiating \eqref{cottonid}, we have
		\[
		0=C_{ijk,l}R_{ij}+C_{ijk}R_{ij,l};
		\]
		tracing with respect to $k$ and $l$ and using the Bach-flatness
		hypothesis, we obtain
		\begin{equation} \label{cottonderid}
			C_{ijk}R_{ij,k}=0.
		\end{equation}
		Now, by the properties 
		of $\operatorname{C}$, \eqref{cottonderid} becomes
		\[
		0=C_{ijk}R_{ij,k}=\dfrac{1}{2}C_{ijk}(R_{ij,k}-R_{ik,j})=\dfrac{1}{2}\abs{\operatorname{C}}^2:
		\]
		hence $\operatorname{C}\equiv 0$ on $M$, which means that
		$(M,g)$ is locally conformally flat.
	\end{proof}
	\begin{rem}
		It can be easily shown that, on every closed Riemannian manifold
		$(M,g)$ of dimension $n\geq 3$, $C_{ijk,k}=0$ on $M$ if and only if 
		$\operatorname{C}\equiv0$ on $M$. Indeed, 
		\[
		0=\int_MC_{ijk,k}P_{ij}dV_g=-\int_MC_{ijk}P_{ij,k}dV_g=
		-\dfrac{1}{2}\int_MC_{ijk}(P_{ij,k}-P_{ik,j})dV_g=-\dfrac{1}{2}\int_M\abs{\operatorname{C}}^2dV_g.
		\]
	\end{rem}
	
	\section{The Bach functional} \label{functionalsection}
	Let $(M,g)$ be a closed Riemannian manifold $n\geq 4$. We recall that the Bach tensor 
	$\bach$ is locally
	defined as
	\begin{equation}
		B_{ij}=\dfrac{1}{n-3}W_{ikjl,lk}+\dfrac{1}{n-2}R_{kl}W_{ikjl}=
		\dfrac{1}{n-2}\pa{C_{ijk,k}+R_{kl}W_{ikjl}}.
	\end{equation}
	Let $u\in\cinf$ and let us consider a conformal change of the
	metric $g$
	\[
	\tilde{g}=e^{2u}g;
	\]
	if $n=4$, the
	Bach tensor is conformally covariant, i.e.
	\begin{equation} \label{bachconfchange}
		e^{4u}\tilde{B}_{ij}=B_{ij},
	\end{equation}
	with respect to a local orthonormal coframe. This allows
	us to write the conformal change of the components of 
	$\nabla\bach$ as follows:
	\begin{equation} \label{nablabachconfchange}
		e^{5u}\tilde{B}_{ij,k}=B_{ij,k}-
		u_iB_{jk}-u_jB_{ik}-4u_kB_{ij}+u_t\pa{
			B_{it}\del{jk}+B_{jt}\del{ik}}.
	\end{equation}
	By a straightforward computation, from \eqref{nablabachconfchange}
	we can derive the formula for the transformed squared norm
	of $\nabla\bach$:
	\begin{equation} \label{squarenablabachconf}
		e^{10u}\abs{\nabla\tilde{\bach}}_{\tilde{g}}^2=
		\abs{\nabla\bach}^2_g-4u_iB_{ij,k}B_{jk}-8u_kB_{ij,k}B_{ij}+
		18\abs{\nabla u}^2_g\cdot\abs{\bach}^2_g+
		8u_pu_qB_{ip}B_{iq}.
	\end{equation}
	Let us locally define the tensor $\operatorname{A}$ as 
	\[
	A_i:=B_{ij,k}B_{jk};
	\]
	then, \eqref{squarenablabachconf} can be rewritten as
	\begin{equation} \label{squarenablabachconf2}
		e^{10u}\abs{\nabla\tilde{\bach}}_{\tilde{g}}^2=
		\abs{\nabla\bach}^2_g-4\langle\operatorname{A}^{\#},\nabla u\rangle
		-4\langle\nabla\abs{\bach}_g^2,\nabla u\rangle
		+18\abs{\nabla u}^2_g\cdot\abs{\bach}^2_g+
		\bach^2\pa{\nabla u,\nabla u},
	\end{equation}
	where $\bach^2=\bach\circ\bach$ and 
	$\operatorname{A}^{\#}\in\mathfrak{X}(M)$ is
	the only smooth vector field such that 
	$\langle\operatorname{A}^{\#},X\rangle=\operatorname{A}(X)$, for every $X\in\mathfrak{X}(M)$.  
	
	Our goal is to show the existence of critical metrics for a
	Riemannian functional which involves the covariant derivative of
	the Bach tensor, following the argument exploited in \cite{cmmp} to
	prove the existence of weak harmonic Weyl metrics on every closed
	$4$-manifold. We define the following functional:
	\begin{equation} \label{bachfunct}
		\mathcal{B}: 
		g\longmapsto \mathrm{Vol}_g(M)^{\frac{3}{2}}\int_M\abs{\nabla\bach_g}_g^2dV_g,
	\end{equation}
	where $\mathrm{Vol}_g(M)$ is the volume of $M$ with respect to $g$.
	It is clear that \eqref{bachfunct} is scale invariant, although, by 
	\eqref{squarenablabachconf2}, it is not conformally invariant. 
	We say that a metric $g$ is \emph{weak Bach-parallel} if it is a
	critical point for the functional $\mathcal{B}$. 
	
	We want to study critical metrics for the functional 
	\eqref{bachfunct}: in order to do so, we derive the Euler-Lagrange
	equations, computing the first variation of the functional in the
	following way. Let
	\[
	u:M\times\RR\longra\RR
	\]
	be a smooth function such that $u(x,0)=0$ and 
	$\frac{du}{dt}(x,0)=U(x)$, where $U\in\cinf$: for every $t\in\RR$,
	we consider the conformal change $\tilde{g}(t)=e^{2u(x, t)}g$ and
	we obtain the following
	\begin{proposition}
		Let $(M,g)$ be a closed Riemannian four-manifold such that
		$g$ is weak Bach-parallel. Then, the Euler-Lagrange equation
		\begin{equation} \label{eulerlagrange}
			\dfrac{1}{2}\Delta\abs{\bach}_g^2=
			\dfrac{3}{4}\abs{\nabla\bach}_g^2-\dfrac{1}{2}
			\diver_g\operatorname{A}-\dfrac{3}{4V}\int_M\abs{\nabla\bach}_g^2
			dV_g
		\end{equation}
		is satisfied, where $V=\mathrm{Vol}_g(M)$.
	\end{proposition}
	\begin{proof}
		The transformed volume form $dV_{\tilde{g}(t)}$ under the
		conformal change is given by
		\[
		dV_{\tilde{g}(t)}=e^{4u}dV_g,
		\]
		which implies that
		\[
		\mathrm{Vol}_{\tilde{g}(t)}(M)=\int_M e^{4u}dV_g.
		\]
		By \eqref{squarenablabachconf} and \eqref{squarenablabachconf2},
		we can compute
		\begin{align*}
			\mathcal{B}(\tilde{g}(t))&=
			\pa{\mathrm{Vol}_{\tilde{g}(t)}(M)}^{\frac{3}{2}}\int_M\abs{\nabla\tilde{\bach}}_{\tilde{g}(t)}^2dV_{\tilde{g}(t)}=\\
			&=\pa{\int_Me^{4u(x,t)}dV_g}^{\frac{3}{2}}
			\int_Me^{-6u(x,t)}\left[\abs{\nabla\bach}_g^2-
			4u_i(x,t)A_i-4u_k(x,t)\pa{\abs{\bach}_g^2}_k+\right.\\
			&+\left.
			18\abs{\nabla u}_g^2(x,t)\cdot\abs{\bach_g^2}+
			8u_p(x,t)u_q(x,t)B_{ip}B_{iq}\right]dV_g=\\
			&=\pa{\int_Me^{4u(x,t)}dV_g}^{\frac{3}{2}}
			\int_Me^{-6u(x,t)}\left[\abs{\nabla\bach}_g^2-
			4\langle\operatorname{A}^{\#},\nabla u\rangle-
			4\langle\nabla\abs{\bach}_g^2,\nabla u\rangle+\right.\\
			&+\left.18\abs{\nabla u}_g^2(x,t)\cdot\abs{\bach_g^2}+
			8\bach^2\pa{\nabla u,\nabla u}\right]dV_g.
		\end{align*}
		Now, recalling the properties of $u$, we obtain
		\begin{align*}
			\dfrac{d}{dt}\restrict{t=0}\mathcal{B}(\tilde{g}(t))&=
			6\pa{\mathrm{Vol}_g(M)}^{\frac{1}{2}}\int_M\abs{\nabla\bach}_g^2
			dV_g\int_M U \, dV_g+\\
			&-\pa{\mathrm{Vol}_g(M)}^{\frac{3}{2}}
			\int_M\sq{6\abs{\nabla\bach}_g^2 U+4U_iA_i-4U_k\pa{\abs{
						\bach}_g^2}_k}dV_g;
		\end{align*}
		integrating by parts, we get
		\begin{align*}
			\dfrac{d}{dt}\restrict{t=0}\mathcal{B}(\tilde{g}(t))&=
			6\pa{\mathrm{Vol}_g(M)}^{\frac{1}{2}}\int_M\abs{\nabla\bach}_g^2
			dV_g\int_M U \, dV_g+\\
			&+\pa{\mathrm{Vol}_g(M)}^{\frac{3}{2}}
			\int_M\sq{6\abs{\nabla\bach}_g^2+4\diver_g\operatorname{A}+
				4\Delta\abs{\bach}_g^2}U\, dV_g.
		\end{align*}
		Since a critical point for \eqref{bachfunct} in the conformal
		class of $g$ is a metric such that
		\[
		\dfrac{d\mathcal{B}(\tilde{g}(t))}{dt}\restrict{t=0}=0, \, 
		\forall U\in\cinf,
		\]
		we immediately obtain \eqref{eulerlagrange}. 
	\end{proof}
	\begin{rem}
		We can locally rewrite the divergence of $\operatorname{A}$ 
		exploiting the commutation rules for the second covariant
		derivatives of symmetric $(0,2)$-tensors 
		(see \cite[Chapter 2]{cmbook}): namely, we have that
		\[
		B_{ij,kl}-B_{ij,lk}=
		B_{pj}R_{pikl}+B_{ip}R_{pjkl}.
		\]
		Therefore, by the fact that $\bach$ is divergence-free
		on every four-manifold, we can compute
		\[
		\diver_g\operatorname{A}=A_{ii}=
		(B_{ij,k}B_{jk})_i=B_{ij,ki}B_{jk}+B_{ij,k}B_{jk,i}=
		B_{jk}B_{jl}R_{kl}+B_{il}B_{jk}R_{ljki}+B_{ij,k}B_{jk,i}.
		\]
	\end{rem}
	Now, in order to show the existence of weak Bach-parallel metrics, 
	we need to fix a reference Riemannian metric such that the associated
	Bach tensor nowhere vanishes: the existence of such metrics on any
	closed smooth four-manifold is proven in \cite{cdmbach}. 
	Hence, let $g$ be a Riemannian metric on $M$ such that 
	$\abs{\bach}_{g}\not\equiv 0$ on $M$ and let
	$v=e^{-3u}$ for some $u\in\cinf$. 
	If we consider the conformal change $\tilde{g}=v^{-\frac{2}{3}}g$,
	we can rewrite \eqref{squarenablabachconf2} as
	\begin{equation}
		\abs{\nabla\tilde{\bach}}_{\tilde{g}}^2=
		v^{\frac{10}{3}}\abs{\nabla\bach}_g^2-
		\dfrac{4}{3}v^{\frac{7}{3}}
		\pa{\langle\operatorname{A}^{\#},\nabla v\rangle+\langle
			\nabla\abs{\bach}_g^2,\nabla v\rangle}+
		v^{\frac{4}{3}}\pa{18\abs{\nabla v}_g^2\cdot\abs{\bach}_g^2+
			8\bach^2(\nabla v,\nabla v)}.
	\end{equation}
	Then, we have
	\begin{align*}
		\mathcal{B}(v)&:=\mathcal{B}(v^{-\frac{2}{3}}g)=\pa{\int_M v^{-\frac{4}{3}}dV_g}^{\frac{3}{2}}
		\Biggl\{\int_Mv^{2}\Big[\abs{\nabla \bach}_g^2+\frac{4}{3v}\langle \operatorname{A}^{\#},\nabla v\rangle\\
		&+\frac{4}{3v}\langle \nabla \abs{\bach}^2_g,\nabla v\rangle
		+\frac{2}{v^2}\abs{\nabla v}^2_g\abs{\bach}^2_g+\frac{8}{9v^2}\bach^2(\nabla v,\nabla v)\Big]dV_g\Biggr\};
	\end{align*}
	integrating by parts, we obtain
	\begin{align} \label{bachfunctv}
		\mathcal{B}(v)&=\pa{\int_M v^{-\frac{4}{3}}dV_g}^{\frac{3}{2}}\left\{
		\int_M -2v\diver_g\pa{\abs{\bach}_g^2\nabla v}
		-\dfrac{8}{9}v\diver_g\pa{\operatorname{T}_v}+\right.\\
		&\left.-\dfrac{2}{3}v^2\pa{\diver_g\operatorname{A}+
			\Delta\abs{\bach}_g^2}+v^2\abs{\nabla\bach}_g^2 dV_g\right\}, 
		\notag
	\end{align}
	where
	\begin{equation} \label{divbachsquare}
		\operatorname{T}_f(X):=\bach^2(\nabla f, X), \quad
		\mbox{ for every } f\in\cinf, \mbox{ } X\in\mathfrak{X}(M).
	\end{equation} 
	Let $\tilde{v}:=v+t\phi$, where $\phi $ is a positive smooth function on $M$; therefore, we obtain
	\begin{align*}
		\frac{d}{dt}|_{t=0}\mathcal{B}(\tilde{v})&=-2\pa{\int_M v^{-\frac{4}{3}}dV_g}^{-1}\int_Mv^{-\frac{7}{3}}\phi\: dV_g\: \mathcal{B}(v)+\\
		&+\pa{\int_Mv^{-\frac{4}{3}}dV_g}^{\frac{3}{2}}\int_M \left[2v\phi\abs{\nabla \bach}_g^2-\dfrac{4}{3}v\phi\pa{\diver_g\operatorname{A}+\Delta\abs{\bach}_g^2+}\right.	\\
		&\left.-\dfrac{8}{9}\pa{\phi\diver_g(\operatorname{T}_v)+
			v\diver_g(\operatorname{T}_\phi)}-2\pa{
			\phi\diver_g\pa{\abs{\bach}_g^2\nabla v}+
			v\diver_g\pa{\abs{\bach}_g^2\nabla \phi}}
		\right]dV_g.
	\end{align*}
	Integrating by parts we deduce
	\begin{align*}
		\frac{d}{dt}|_{t=0}\mathcal{B}(\tilde{v})&=-2\pa{\int_M v^{-\frac{4}{3}}dV_g}^{-1}\int_Mv^{-\frac{7}{3}}\phi\: dV_g\: \mathcal{B}(v)+\\
		&+\pa{\int_Mv^{-\frac{4}{3}}dV_g}^{\frac{3}{2}}\int_M\Big[-\frac{4}{3}v\pa{\diver_g\operatorname{A}+
			\Delta\abs{\bach}^2_g-\frac{3}{2}\abs{\nabla \bach}^2_g}+\\
		&-\diver_g\pa{4\abs{\bach}^2_g\nabla v+\dfrac{16}{9}\operatorname{T}_v^{\#}}\Big]\phi\:dV_g.
	\end{align*}
	Imposing that $v$ is a critical point for the functional $v \mapsto \mathcal{B}(v)$, i.e.
	\[\frac{d}{dt}|_{t=0}\mathcal{B}(v+t\phi)=0\]
	for all $0<\phi \in C^{\infty}(M)$, we obtain the Euler-Lagrange equation
	\begin{equation} \label{pdetosolve}
		-\diver_g\pa{a\nabla v+\dfrac{4}{9}\operatorname{T}^{\#}_v}+cv=\dfrac{\lambda(v)}{v^{\frac{7}{3}}}
	\end{equation}
	where 
	\[
	a:=\abs{\bach}_g^2, \mbox{ } 
	c:=-\dfrac{1}{3}\pa{\diver_g\operatorname{A}+\Delta\abs{\bach}_g^2-
		\dfrac{3}{2}\abs{\nabla\bach}_g^2}, \mbox{ } 
	\lambda(v):=\frac{\mathcal{B}(v)}{2\pa{\int_M 
			v^{-\frac{4}{3}}dV_g}^{\frac{5}{2}}}.
	\]
	Hence, to prove the existence of a minimizer of $\mathcal{B}(g)$ is
	sufficient to show that \eqref{pdetosolve} admits a smooth, positive
	solution. 
	
	\section{Proof of Theorem \ref{weakbachparallel}}\label{section proof main thm}
	
	Let us define the second order differential operator
	\begin{equation} \label{bachoperator}
		Lv:=-\diver_g\pa{a\nabla v+\dfrac{4}{9}\operatorname{T}^{\#}_v}+cv,
	\end{equation}
	where $a$ and $c$ are as before and $\operatorname{T}_v$ is defined
	in \eqref{divbachsquare}. By the existence of Riemannian metrics with nowhere vanishing Bach tensor proven in \cite{cdmbach}, 
	we fix a Riemannian metric $g$ on $M$ such that 
	$\abs{\bach}_g\neq 0$ on $M$: since, locally, 
	$(\operatorname{T}^{\#}_v)^i=(B^2)^i_{j}v^j$, we have that the matrix of
	the highest order coefficients of \eqref{bachoperator} is
	given by
	\begin{equation} \label{ellipticmatrixcomp}
		G^i_{j}:=\abs{\bach}_g^2\delta^i_{j}+\dfrac{4}{9}(B^2)^i_{j}.
	\end{equation}
	Since the $(0,2)$-version of $\bach^2$ is symmetric, it is immediate to see that $L$ is self-adjoint; furthermore, by our choice of $[g]$, 
	the matrix $(B^2)^i_{j}$ is positive
	semidefinite, which means that $G^i_{j}$ is uniformly positive
	definite. Since the operator
	$L'v=-\diver_g(a\nabla v)$ is uniformly elliptic and $a>0$ on $M$, $L$ is also uniformly elliptic.
	
	Our strategy, based upon the argument in \cite{cmmp}, is to search
	for weak solutions of \eqref{pdetosolve}: namely, we define
	\[
	H(M)=\set{u\in H^1(M): \mbox{ } u>0 \mbox{ a.e.}, \int_Mu^{-\frac{4}{3}}dV_g<+\infty}
	\]
	and the minimum of the Bach functional in $H(M)$ as
	\[
	\mathfrak{B}=\inf_{u\in H(M)}\mathcal{B}(u).
	\]
	Once the existence of $v\in H(M)$ such that 
	$\mathcal{B}(v)=\mathfrak{B}$
	is established, 
	we want to show that
	$\mathcal{B}(v)=\ol{\mathfrak{B}}$, where
	\[
	\ol{\mathfrak{B}}=\inf_{\substack{u\in\cinf \\ u>0}}
	\mathcal{B}(u).
	\] 
	
	First, let us define the \emph{principal eigenvalue} of $L$ as
	\begin{equation} \label{firsteigen}
		\lambda_1=\inf_{\substack{u\in H^1(M) \\ u\not\equiv0}}
		\mathfrak{R}(u):=\inf_{\substack{u\in H^1(M) \\ u\not\equiv0}}
		\dfrac{\int_M uLu \, dV_g}{\int u^2dV_g},
	\end{equation}
	where $\mathfrak{R}(u)$ denotes the \emph{Rayleigh quotient}
	and
	\begin{equation} \label{uludef}
		\int_MuLu \, dV_g=\int_M\pa{a\abs{\nabla u}^2+
			\dfrac{4}{9}{\bach}^2(\nabla u,\nabla u)+cu^2}dV_g:
	\end{equation}
	it is well-known that $L$ admits a smooth, positive eigenfunction 
	$\phi_1$ relative to $\lambda_1$, i.e. such that 
	$L\phi_1=\lambda_1\phi_1$. 
	
	We point out that, \emph{a priori}, the zero-order coefficient 
	$c$ in \eqref{bachoperator} does not have a given sign: hence, we might not be able to conclude the
	non-negativity of the principal eigenvalue of $L$. 
	However, it is immediate to show that, by
	\eqref{bachfunctv} and \eqref{firsteigen}, for every $u\in H(M)$ we have  
	\begin{equation} \label{bachfunctrayleigh}
		\mathcal{B}(u)=2\pa{\int_M u^{-\frac{4}{3}}dV_g}^{\frac{3}{2}}
		\int_MuLu \, dV_g;
	\end{equation}
	this equality allows us to prove the following
	\begin{lemma} \label{pinchinglemma}
		The minimum $\mathfrak{B}$ satisfies
		\[
		2\mathrm{Vol}_g(M)^{\frac{5}{2}}\lambda_1\leq\mathfrak{B}\leq
		2\dfrac{\int_M\phi_1^2dV_g}{\pa{\int_M\phi_1^{-\frac{4}{3}}dV_g}^{-\frac{3}{2}}}\lambda_1.
		\]
	\end{lemma}
	\begin{proof}
		Let $V=\mathrm{Vol}_g(M)$ again; first, Jensen's inequality tells us that, for every 
		$u\in H(M)$,
		\[
		\pa{\int_M u^2dV_g}^{-\frac{2}{3}}\leq\int_M
		u^{-\frac{4}{3}}dV_g\cdot V^{-\frac{5}{3}},
		\]
		hence,
		\[
		\pa{\int_M u^2dV_g}^{-1}\leq\pa{\int_M
			u^{-\frac{4}{3}}dV_g}^{\frac{3}{2}}\cdot V^{-\frac{5}{2}}.
		\]
		By \eqref{firsteigen} and 
		\eqref{bachfunctrayleigh}, we obtain
		\begin{align*}
			\lambda_1&\leq\dfrac{\int_MuLu \, dV_g}{\int_Mu^2dV_g}\leq
			\dfrac{1}{2}\pa{\int_M
				u^{-\frac{4}{3}}dV_g}^{-\frac{3}{2}}\mathcal{B}(u)
			\pa{\int_M
				u^{-\frac{4}{3}}dV_g}^{\frac{3}{2}}\cdot V^{-\frac{5}{2}}=\\
			&=\dfrac{1}{2}\mathcal{B}(u)\cdot V^{-\frac{5}{2}}:
		\end{align*}
		since this holds for every $u\in H(M)$, it implies the first inequality in the statement. 
		
		On the other hand, by \eqref{bachfunctrayleigh} and the
		definition of $\mathfrak{B}$
		we easily deduce that
		\[
		\mathfrak{B}\leq\mathcal{B}(\phi_1)=2
		\mathfrak{R}(\phi_1)\dfrac{\int_M\phi_1^2dV_g}{\pa{\int_M
				\phi_1^{-\frac{4}{3}}dV_g}^{-\frac{3}{2}}}
		\]
		and, since $\mathfrak{R}(\phi_1)=\lambda_1$, it concludes the
		proof.
	\end{proof}
	
	By definition \eqref{bachfunct}, we know that
	$\mathcal{B}(v)\geq 0$ for every smooth, positive
	function $v$. We rewrite \eqref{bachfunctv} as
	\[
	\mathcal{B}(v)=\pa{\int_Mv^{-\frac{4}{3}}dV_g}^\frac{3}{2}\int_M2\sq{\langle G\nabla v, \nabla v\rangle+ cv^2}dV_g,
	\]
	for every $v\in H(M)$, where the components of $G$ are defined in \eqref{ellipticmatrixcomp}.
	Since both $G$ and $c$ are smooth, they are also bounded on $M$: hence, by the fact that 
	$v$ and $\nabla v$ are in $L^2(M)$, we obtain
	\[
	\abs{\int_M\sq{\langle G\nabla v, \nabla v\rangle+ cv^2}dV_g}\leq C\norm{G}_{\infty}\norm{c}_{\infty}\norm{v}_{L^2(M)}^2\norm{\nabla v}_{L^2(M)}^2<+\infty
	\]
	for every $v\in H^1(M)$. Hence, 
	\[
	I(v)\coloneqq\int_M\sq{\langle G\nabla v, \nabla v\rangle+ bv^2}dV_g
	\]
	is a continuous quadratic form on 
	$H^1(M)$, meaning that, for every sequence 
	$\{v_n\}_n\subset H^1(M)$ converging to 
	$v\in H^1(M)$, 
	\[
	\lim_{n\to+\infty}I(v_n)=I(v).
	\]
	By a standard density argument, we can define a sequence of positive, smooth functions $v_n$ 
	such that $v_n\to v$ in $H^1(M)$ 
	(for instance, using heat semigroups, see e.g.
	\cite{grigbook, grigoryan, rosenberg}) and exploit the continuity of $I$ to show that 
	\[
	I(v)=\lim_{n\to +\infty}I(v_n)\geq 0.
	\]
	Furthermore, since $v\in H(M)$, 
	\[
	0<\int_Mv^{-\frac{4}{3}}dV_g<+\infty,
	\]
	which means that 
	\[
	\mathcal{B}(v)=\int_Mv^{-\frac{4}{3}}dV_g \cdot I(v)\geq 0
	\]
	for every $v\in H(M)$ and, in particular, 
	$\mathfrak{B}\geq 0$. 
	Using this fact, we obtain the following consequence
	of Lemma \ref{pinchinglemma}:
	\begin{cor} \label{eigenvpositive}
		The principal eigenvalue $\lambda_1$ of $L$ is non-negative. 
	\end{cor}
	Now, in view of Lemma \ref{pinchinglemma}, if $\mathfrak{B}=0$ the
	proof of Theorem \ref{weakbachparallel} is trivial: indeed, 
	this would imply that $\lambda_1=0$, which means that, 
	by \eqref{bachfunctrayleigh},
	\[
	\mathcal{B}(\phi_1)=
	2\pa{\int_M\phi_1^{-\frac{4}{3}}dV_g}^{\frac{3}{2}}
	\int_M\phi_1L\phi_1dV_g=0=\mathfrak{B}.
	\]
	Hence, $\phi_1$ attains the minimum of $\mathcal{B}$ in the
	given conformal class.
	
	Now, let $\mathfrak{B}>0$, which, by Lemma \ref{pinchinglemma} and Corollary \ref{eigenvpositive}, 
	implies that $\lambda_1>0$: we divide the proof in steps, for
	the sake of simplicity. 
	
	\addtocontents{toc}{\protect\setcounter{tocdepth}{1}}
	
	\subsection*{Step 1.}
	
	Let us consider a sequence of functions 
	$\{v_n\}_{n\in\NN}\subset H(M)$ such that $\mathcal{B}(v_n)\to
	\mathfrak{B}$ as $n\to +\infty$; 
	exploiting the scale-invariance of $\mathcal{B}$, we can always
	assume that $\int_M v_n^{-\frac{4}{3}}dV_g=1$ for every $n$.
	We can immediately observe that there exists $N\in \NN$
	\[
	0<\mathfrak{B}\leq\mathcal{B}(v_n)\leq \mathfrak{B}+1,
	\]
	for every $n\geq N$. 
	By \eqref{firsteigen} and \eqref{uludef}, 
	it is clear that
	\begin{align*}
		\lambda_1||v_n||_{L^2(M)}^2&\leq \int_M\pa{a\abs{\nabla v_n}^2+\dfrac{4}{9}\bach^2(\nabla v_n,\nabla v_n)
			+cv^2_n}dV_g=\\
		&=\pa{\int_Mv_n^{-\frac{4}{3}}dV_g}^{\frac{3}{2}}
		\int_M\pa{a\abs{\nabla v_n}^2+
			\dfrac{4}{9}\bach^2(\nabla v_n,\nabla v_n)+cv^2_n}dV_g=
		\mathcal{B}(v_n)\leq \mathfrak{B}+1:
	\end{align*}
	this means that $\{v_n\}_{n\in \NN}$ is bounded in $L^2(M)$.
	Now we need to prove that $\{\nabla v_n\}_{n\in\NN}$ is bounded
	in $L^2(M)$, which, in turn, would imply that $\{v_n\}_{n\in\NN}$
	is bounded in $H^1(M)$. By the fact that $a\neq 0$ and that 
	$\bach^2$ is positive semidefinite when regarded as a quadratic 
	form, we have that there exists $N\in \NN$ such that, for every
	$n\geq N$, 
	\begin{align*}
		\sq{\min_{p\in M}a(p)}\int_M\abs{\nabla v_n}^2dV_g&\leq
		\int_Ma\abs{\nabla v_n}^2dV_g=\mathcal{B}(v_n)-
		\int_M\pa{\dfrac{4}{9}B^2(\nabla v_n,\nabla v_n)+cv_n^2}dV_g\leq\\
		&\leq \mathcal{B}(v_n)-\int_Mcv^2_ndV_g\leq
		\mathcal{B}(v_n)+||c||_{L^{\infty}(M)}||v_n||_{L^2(M)}^2\\
		&\leq
		\mathfrak{B}+1+||c||_{L^{\infty}(M)}
		\dfrac{\mathfrak{B}+1}{\lambda_1}
	\end{align*}
	and the claim is proven (the last inequality is derived by the
	upper bound on $\lambda_1||v_n||_{L^2(M)}^2$).
	Hence, the boundedness of the sequence in $H(M)$ implies that,
	up to passing to a subsequence, $\{v_n\}_{n\in \NN}$ weakly
	converges in $H(M)$, i.e. there exists $v\in H(M)$ such that
	\[
	v_n\rightharpoonup v \quad \mbox{in } H^1(M),
	\]
	which, by compactness of $H^1(M)\hookrightarrow L^2(M)$, implies 
	that, up to subsequences, 
	\[
	v_n\longrightarrow v \quad \mbox{in } L^2(M), \quad 
	v_n\longra v \quad \mbox{a.e. on } M.
	\]
	In fact, we show that the convergence is in $H(M)$: indeed, 
	Fatou's
	Lemma implies that
	\[
	\int_Mv^{-\frac{4}{3}}dV_g\leq\liminf_{n\to+\infty}
	\int_Mv_n^{-\frac{4}{3}}dV_g=1<+\infty,
	\]
	since $v_n^{-\frac{4}{3}}\to v^{-\frac{4}{3}}$ a.e. on $M$, and
	from this we can also deduce that $v>0$ a.e. on $M$, which means
	that $v\in H(M)$. 
	
	\addtocontents{toc}{\protect\setcounter{tocdepth}{2}}
	
	\addtocontents{toc}{\protect\setcounter{tocdepth}{1}}
	\subsection*{Step 2.}
	Now, we need to show that $v$ is in fact a weak solution of
	\eqref{pdetosolve}, i.e. that $\mathcal{B}(v)=\mathfrak{B}$. 
	By \eqref{uludef}, \eqref{bachfunctv} and the estimate for the
	integral of $v^{-\frac{4}{3}}$ we have
	\begin{align*}
		\mathfrak{B}\leq\mathcal{B}(v)&=2\pa{\int_M
			v^{-\frac{4}{3}}dV_g}^{\frac{3}{2}}\int_M
		\pa{a\abs{\nabla v}^2+
			\dfrac{4}{9}B^2(\nabla v,\nabla v)+cv^2}dV_g\leq\\
		&\leq \int_M
		2\pa{a\abs{\nabla v}^2+\dfrac{4}{9}B^2(\nabla v,\nabla v)+cv^2}dV_g.
	\end{align*}
	By exploiting the weak convergence in $H^1(M)$ and the convergence
	in $L^2(M)$, we can show that
	\[
	\mathfrak{B}\leq2\liminf_{n\to+\infty}\int_M
	\pa{a\abs{\nabla v_n}^2+
		\dfrac{4}{9}B^2(\nabla v_n,\nabla v_n)+cv_n^2}dV_g=
	\liminf_{n\to+\infty}\mathcal{B}(v_n)=\mathfrak{B},
	\]
	which means that 
	\[
	\mathcal{B}(v)=\mathfrak{B}.
	\] 
	\addtocontents{toc}{\protect\setcounter{tocdepth}{2}}
	\addtocontents{toc}{\protect\setcounter{tocdepth}{1}}
	\subsection*{Step 3.} Finally, we need to prove that $v$ is a weak solution of \eqref{pdetosolve} and that it is actually
	a smooth positive function, which would imply
	\[
	\mathfrak{B}=\ol{\mathfrak{B}}=\mathcal{B}(v)
	\]
	and, therefore, would guarantee the existence of the conformal 
	factor. We observe that, \emph{a priori}, 
	$v^{-1}\in L^{\frac{4}{3}}(M)\setminus L^{\frac{7}{3}}(M)$: therefore, we need to 
	establish a better regularity for $v$ in order to
	proceed. 
	
	More precisely, we can prove that 
	$v$ is a weak supersolution of the PDE $Lu=0$:
	i.e. we show the following
	\begin{proposition} \label{weaksupersolution}
		Let $v\in H(M)$ such that $\mathcal{B}(v)=\mathfrak{B}>0$. Then, 
		$Lv\geq 0$ in the weak sense, i.e., for every $\phi\in C^1(M)$
		such that $\phi\geq 0$, 
		\[
		\int_M vL\phi\, dV_g=\int_M \left(\langle G\nabla v, \nabla \phi\rangle + cv\phi\right)dV_g \geq 0.
		\]
	\end{proposition}
	\begin{proof}
		Assume that there exists $\phi$ such that 
		$\int_M vL\phi \, dV_g<0$ and define, for every $u\in H^1(M)$,
		\[
		Q(u)\coloneqq\int_M uLu \, dV_g.
		\]
		Arguing as in Step 1., we know that 
		$Q(u)\geq\lambda_1 ||u||_{L^2(M)}^2$: for 
		$t>0$ small enough we consider
		the increment
		\begin{align*}
			&\mathcal{B}(v+t\phi)-\mathcal{B}(v)=\\
			&=\sq{\pa{\int_M\pa{v+t\phi}^{-\frac{4}{3}}dV_g}^{\frac{3}{2}}-
				\pa{\int_Mv^{-\frac{4}{3}}dV_g}^{\frac{3}{2}}}Q(v+t\phi)+
			\pa{\int_Mv^{-\frac{4}{3}}dV_g}^{\frac{3}{2}}[Q(v+t\phi)-Q(v)].
		\end{align*}
		Since $t>0$ and $Q(u)\geq 0$ for every $u\in H^1(M)$, we have
		\[
		\mathcal{B}(v+t\phi)-\mathcal{B}(v)\leq 
		\pa{\int_Mv^{-\frac{4}{3}}dV_g}^{\frac{3}{2}}[Q(v+t\phi)-Q(v)].
		\]
		Expanding the term in square brackets, as 
		$t\to 0$ we get
		\[
		Q(v+t\phi)-Q(v)=Q'(v)[\phi]t+o(t),
		\]
		where, by the variational computations in Section 
		\ref{functionalsection},
		\[
		Q'(v)[\phi]=2\int_M vL\phi \, dV_g,
		\]
		which is negative by our assumption. Hence, for
		sufficiently small $t>0$,
		\[
		\mathcal{B}(v+t\phi)-\mathcal{B}(v)\leq
		\pa{\int_Mv^{-\frac{4}{3}}dV_g}^{\frac{3}{2}}Q'(v)[\phi]t+o(t)<0,
		\]
		i.e. $\mathcal{B}(v+t\phi)<\mathcal{B}(v)$. However, this is 
		a contradiction: indeed, since $t>0$ and $\phi\geq 0$, 
		$v+t\phi\in H(M)$, hence we would have that there exists 
		$u\in H(M)$ such that
		\[
		\mathcal{B}(u)<\mathfrak{B}. 
		\]
	\end{proof}
	\noindent
	Now, we can exploit the following version of the strong maximum
	principle:
	\begin{proposition}\label{p-maxprinc}
		Let $u\in H^1(M)$ such that $Lu\geq 0$ in the weak sense, with
		$\lambda_1>0$. Then, either $u=0$ a.e. on $M$ or 
		$\operatorname*{ess\,inf}_Mu>0$.	
	\end{proposition}
	\begin{rem}
		We do not provide a proof of Proposition \ref{p-maxprinc}, since it is a standard argument: writing $u$ as the sum of its positive and negative parts, the assumption $\lambda_1>0$ forces $u^-$ to be zero almost everywhere, so that $u$ is non-negative almost everywhere. Then, either $u=0$ almost everywhere, or, as a consequence of the weak Harnack inequality applied to a finite cover of $M$, $\operatorname*{ess\,inf}_Mu>0$. For more details we refer the reader to \cite{giltr, moser}.
	\end{rem}
	Therefore, since $v>0$ a.e., by Proposition \ref{weaksupersolution} we deduce that 
	$\operatorname*{ess\,inf}_Mv>0$: 
	in particular, $v^{-1}\in L^{\infty}(M)$ and, hence, 
	$v^{-1}\in L^p(M)$ for every $p>0$. 
	Since $v$ is such that $\mathcal{B}(v)=\mathfrak{B}$, by the G\^{a}teaux differentiability of 
	$\mathcal{B}$ at $v$, we have that, for every $\phi\in C^1(M)$, 
	\[
	\dfrac{d}{dt}\restrict{t=0}\mathcal{B}(v+t\phi)=0,
	\] 
	i.e. $v$ is a weak solution of 
	\[
	Lv=\dfrac{\mathfrak{B}}{2\left(\int_M v^{-\frac{4}{3}}dV_g\right)^{\frac{3}{2}}}v^{-\frac{7}{3}}=\lambda(v)v^{-\frac{7}{3}},
	\]
	which is exactly \eqref{pdetosolve}.

	Since $\lambda(v)<+\infty$ and $v^{-1}\in L^{\infty}$, the right-hand side of \eqref{pdetosolve} must be a $L^{\infty}(M)$ function: therefore, a standard bootstrapping argument using elliptic regularity theory assures that
	$v$ is smooth and positive. In particular, 
	the weak solution found in Step 3 is smooth and positive everywhere, which means that 
	there exists a critical point $\tilde{g}\in [g]$ 
	of \eqref{bachfunct}.
	\addtocontents{toc}{\protect\setcounter{tocdepth}{2}}
	\addtocontents{toc}{\protect\setcounter{tocdepth}{1}}
	\subsection*{Step 4.} Now, we need to prove that, in the chosen
	conformal class $[g]$, the minimizer $\tilde{g}$ 
	identified by $v$ is unique,
	up to multiplication by a positive constant: indeed,
	by the scaling invariance of \eqref{bachfunct}, we cannot expect to obtain a unique solution of 
	\eqref{pdetosolve} in $[g]$. 
	
	First, we highlight the fact that every weak solution 
	of \eqref{pdetosolve} in $H(M)$ must be smooth and 
	positive on $M$. Indeed, let $u$ be such a weak solution, i.e. 
	\[
	\int_M(\langle G\nabla v, \nabla\phi\rangle+
	cv^2)dV_g=\lambda(v)\int_M\phi v^{-\frac{7}{3}}dV_g
	\]
	for every $\phi\in C^1(M)$, $\phi\geq 0$:
	note that, since the left-hand side and $\lambda(v)$ are bounded, we immediately obtain that $\phi v^{-\frac{7}{3}}\in L^1(M)$. 
	In particular, since $\phi\geq 0$ and $v>0$ a.e.,
	we immediately obtain that $Lv\geq 0$ in the weak sense: thus, by Proposition \ref{weaksupersolution} and Proposition \ref{p-maxprinc}, we can apply the
	same argument we performed at the end of Step 3 to conclude that $v$ is positive and smooth. 
	
	By these considerations,
	suppose that $u_1$ and $u_2$ are two smooth, positive solutions of \eqref{pdetosolve}:
	our claim is that there exists $\beta>0$ such that
	\[
	u_1=\beta u_2\quad\text{in } M.
	\]	
	First, we observe that the claim is trivial if
	$\mathcal{B}(u_1)=\mathcal{B}(u_2)=0$: indeed,
	this would imply that $\ol{\mathfrak{B}}=0$ and, by Lemma \ref{pinchinglemma}, that $\lambda_1=0$. By simplicity of the principal eigenvalue, that 
	there exist $\beta_1,\beta_2>0$ such that
	\[
	u_i=\beta_i\phi_1,
	\]
	where $\phi_1$ is a positive, smooth eigenfunction
	of $L$ associated to $\lambda_1$. 
	
	Then, we claim that, if $\mathcal{B}(u_1)>0$,
	then $\mathcal{B}(u_2)>0$. Indeed, assume that
	$\mathcal{B}(u_2)=0$ and define
	\[
	\psi:=\mu u_1, \qquad \mu:=
	\pa{\int_Mu_1^{-\frac{4}{3}}dV_g}^{\frac{3}{4}}
	\pa{\dfrac{2}{\mathcal{B}(u_1)}}^{\frac{3}{10}}:
	\]
	by the definition of $\lambda(v)$ we obtain
	\[
	L\psi=\psi^{-\frac{7}{3}}.
	\]
	We can choose $\alpha>0$ such that 
	\[
	f:=\psi-\alpha u_2\geq 0, \quad 
	\min_M f=0;
	\]
	by continuity of $f$, there exists a minimum point $p_0$ for $f$, at which
	\[
	\psi(p_0)=\alpha u_2(p_0).
	\]
	At $p_0$, we have
	\[
	f(p_0)=\nabla f (p_0)=0, \quad (\hess f)(p_0)\geq 0
	\]
	and, by \eqref{bachoperator}, 
	\[
	Lf=-a\Delta f-\dfrac{4}{9}\langle\bach^2,\hess f\rangle;
	\]
	hence, we immediately obtain 
	\[
	0\geq -a\Delta f - \dfrac{4}{9}\langle\bach^2,\hess f\rangle=Lf=L(\psi-\alpha u_2)=L\psi=\psi^{-\frac{7}{3}}>0
	\]
	at $p_0$, which is impossible. 
	
	Hence, we assume that $\mathcal{B}(u_1), \mathcal{B}(u_2)>0$: we define $\psi$ as before 
	and $\varphi$ as 
	\[
	\varphi:=\nu u_2, \quad \nu:=\pa{\int_Mu_2^{-\frac{4}{3}}dV_g}^{\frac{3}{4}}
	\pa{\dfrac{2}{\mathcal{B}(u_2)}}^{\frac{3}{10}}.
	\]
	Again, let $\alpha>0$ such that 
	\[
	f:=\psi-\alpha\varphi\geq 0, \quad 
	\min_{M}f=0:
	\]
	by definition of $\psi$ and $\varphi$, we obtain
	\[
	\begin{cases}
		L\psi=\psi^{-\frac{7}{3}}\\
		L\varphi=\varphi^{-\frac{7}{3}}
	\end{cases}
	\Longra Lf=\psi^{-\frac{7}{3}}-\alpha\varphi^{-\frac{7}{3}}.
	\]
	If $p_0$ is the mininum point of $f$, 
	we get 
	\[
	0\geq \psi^{-\frac{7}{3}}(p_0)
	-\alpha\varphi^{-\frac{7}{3}}(p_0)=
	\dfrac{1-\alpha^{\frac{10}{3}}}{\alpha^{\frac{7}{3}}
		\varphi^{\frac{7}{3}}(p_0)} \Longra \alpha\geq 1,
	\]
	which implies that 
	\[
	\psi\geq \alpha\varphi\geq \varphi.
	\]
	Now, it is sufficient to exploit the same argument
	exchanging the roles of $\psi$ and $\varphi$ in order 
	to obtain
	\[
	\psi=\varphi \Longra u_1=\beta u_2, \quad 
	\beta:=\dfrac{\nu}{\mu}>0
	\]
	and this proves the claim. 
	
	Recall that, in the previous steps, we showed that
	there exists a smooth, positive solution $v$ of 
	\eqref{pdetosolve} which attains the infimum 
	$\ol{\mathfrak{B}}$: by the uniqueness result and
	the fact that every solution $v\in H(M)$ of 
	\eqref{pdetosolve} is smooth and positive, we can
	state the following
	\begin{proposition}
		Every critical point of the functional 
		$v\longmapsto\mathcal{B}(v)$, where $v\in H(M)$, 
		is a minimum point. In particular, in every 
		conformal class $[g]$ such that $\bach_g\neq 0$ on $M$, there exists a unique minimizer $\tilde{g}$ of $\mathcal{B}(g)$ for conformal variations. 
	\end{proposition}

	\addtocontents{toc}{\protect\setcounter{tocdepth}{2}}
	\section{Some properties of the global minimizers I: products of surfaces}\label{secction global min I}
	
	We want to study some features of 
	absolute minima of the Bach functional defined in 
	\eqref{bachfunct}: in particular, we focus on   
	Bach-parallel metrics, i.e. the ones such that $\nabla\bach\equiv 0$. Obviously, Bach-flat metrics 
	represent global minima of the functional, since they are all
	Bach-parallel: however, we aim at a characterization of 
	Bach-parallel metrics whose Bach tensors do not vanish identically.
	In order to do so, we exploit a modified version of the
	well-known de Rham's splitting Theorem (see e.g. \cite{boubel}):
	\begin{theorem} \label{splittingbach}
		Let $(M,g)$ be a simply connected, complete Riemannian manifold
		of dimension $n$. Assume that there
		exists a symmetric, parallel $(0,2)$-tensor 
		field $\operatorname{T}$. Then there exist $c_1,...,c_k\in\erre$ and Riemannian manifolds $(M_1,g_1),...,(M_k,g_k)$ such that 
		\begin{equation} \label{derhamdecomp}
			(M,g)\cong\Pi_{i=1}^k (M_i,g_i)
		\end{equation}
		and 
		\[\operatorname{T}\restrict{M_i}=c_i g_i.\]
		The isometry is unique up to compositions with isometries of the factors.
	\end{theorem}
	\begin{rem}
		If $M$ is not simply connected, Theorem \ref{splittingbach} holds locally, as shown in \cite{eisen}.
	\end{rem}
	\begin{rem}
		Let $\operatorname{T}_g$ be a
		symmetric, parallel $(0,2)$-tensor defined on $(M,g)$ and $\operatorname{T}_{g_i}$ be the same tensor, but defined on the factor $M_i$ of the decomposition
		\eqref{derhamdecomp}: it is important to observe that, in general, 
		\[
		\operatorname{T}_g\restrict{M_i}\neq \operatorname{T}_{g_i}.
		\]
		For instance, if $\operatorname{T}_g$ is the Bach tensor $\bach_g$, its restrictions to the factors of the decomposition will not be, in general, the Bach tensors of the metrics $g_i$.
	\end{rem}
	
	When $(M,g)$ is Bach-parallel, Theorem \ref{splittingbach} implies that $(M,g)$ locally splits as $\Pi_{i=1}^k (M_i,g_i)$, $k\leq4$ and the splitting is global when $(M,g)$ is simply connected. Therefore, due to Theorem \ref{splittingbach}, we focus on the cases where $(M,g)$ is isometric to the product of two surfaces and a product of a three-dimensional Riemannian manifold $(N,g_N)$ with $(\SS^1,g_{\SS^1})$.
	\subsection{Bach-parallel metrics on a product of surfaces}
	If $(M,g)$ is simply connected and it is isometric to a Riemannian product of surfaces, we obtain a classification result (Theorem \ref{bachpar22}):
	\begin{proof}[Proof of Theorem \ref{bachpar22}]	
		We choose 
		a local orthonormal coframe $\{\theta^1,...,\theta^4\}$ for $\bar{g}=g_1\oplus g_2$,
		where $\{\theta^1,\theta^2\}$ (respectively, $\{\theta^3,\theta^4\}$) is a local coframe for 
		$g_1$ (respectively, for $g_2$). 
		In this case, we adopt the index convention 
		$a,b,c,...=1,2$ and $i,j,k,...=3,4$: for instance, 
		\[
		R_{abcd}=K_1\pa{\delta_{ac}\delta_{bd}-\delta_{ad}\delta_{bc}}, \quad R_{ijkl}=K_2\pa{\delta_{ik}\delta_{jl}-\delta_{il}\delta_{jk}}.
		\]
		By computing the components of the Ricci tensor and 
		the Weyl tensor, we are able to conclude that
		\begin{align} \label{bach22comp}
			B_{ab}&=\dfrac{1}{6}\sq{
				2\Delta_1 K_1\delta_{ab}-2(K_1)_{ab}-\Delta_2K_2\delta_{ab}+\pa{K_1^2-K_2^2}\delta_{ab}};\\
			B_{ai}&=0; \notag\\
			B_{ij}&=\dfrac{1}{6}\sq{
				2\Delta_2 K_2\delta_{ij}-2(K_2)_{ij}-\Delta_1K_1\delta_{ij}+\pa{K_2^2-K_1^2}\delta_{ij}}, \notag
		\end{align}
		where $(K_1)_{ab}$ are the components of the Hessian of $K_1$ and 
		$\Delta_1$ is the Laplace--Beltrami operator with respect to
		$g_1$ (the analogue
		holds for $(K_2)_{ij}$ and $\Delta_2$).
		It is immediate to see that, if both $K_1$ and $K_2$ are constant, $\nabla\bach\equiv 0$: 
		furthermore, it is known that
		\begin{itemize}
			\item $(M,g)$ is an Einstein manifold if and only if 
			$K_i$ is constant for every $i$ and $K_1=K_2$;
			\item $(M,g)$ is a
			locally conformally flat manifold if and only if 
			$K_i$ is constant for every $i$ and $K_1=-K_2$. 
		\end{itemize}
		In both cases, we would have that $(M,g)$ is 
		Bach-flat, contradicting our assumption: hence, 
		$K_1\neq\pm K_2$.
		
		Now, we assume that $K_i$ is not constant for 
		some $i=1,2$: by Theorem \ref{splittingbach}, we know that the restriction of $\bach$ to $N_i$ is a constant multiple of $g_i$, i.e. there exists $c\in\RR$ such that
		\[
		B_{11}=B_{22}=c, \quad B_{33}=B_{44}=-c
		\]
		and $B_{\alpha\beta}=0$ whenever 
		$\alpha\neq\beta$. 
		Without loss of generality, we focus on $N_1$ and we assume that $K_1$ is not constant: using 
		$B_{11}=B_{22}$ and $B_{12}=0$ we obtain
		\[
		(K_1)_{11}=(K_1)_{22}, \quad (K_1)_{12}=0.
		\]
		%
		This means that the \emph{trace-free Hessian} of $K_1$, defined as 
		\[
		\mathring{\hess_1}(K_1):=\hess_1(K_1)-\dfrac{\Delta_1K_1}{2}g_1
		\]
		vanishes at every point $p\in N_1$ (here $\hess_1$ denotes the Hessian with respect
		to the metric $g_1$), i.e. 
		$(N_1,g_1)$
		is a \emph{conformal gradient soliton} with potential function 
		$K_1$. 
		By the work of Tashiro \cite{tashiro}, 
		we know that $\nabla K$ is a conformal vector field, $K$ has at most $2$ critical points and 
		the conformal structure of $(N_1,g_1)$ is classified according to $\#\operatorname{crit}(K_1)$ (see also \cite{catmantmazz}). 
		
		For the sake of simplicity, from now on we will denote $K:=K_1$, $\hess:=\hess_1$ and $\Delta:=\Delta_1$. First, assume that $\#\operatorname{crit}(K)=2$: in this case, $M$ is compact and, by a result due to 
		Bourguignon and Ezin \cite[Theorem II.9]{bourguignonezin}, we have that 
		\[
		0=\int_{N_1} (\nabla K)S \, dV_{g_1}=
		2\int_{N_1}(\nabla K)K \, dV_{g_1}=2\int_{N_1}\abs{\nabla K}^2dV_{g_1},
		\]
		which is impossible, since we are assuming that 
		$K$ is not constant, and this implies that 
		$\#\operatorname{crit}(K)\leq 1$. 
		If $p\in N_1$ is a regular point for $K$ and 
		$K(p)=C$, then there exists an open 
		neighborhood $U$ of the level set 
		$\Sigma=\{K=C\}$, which does not contain any critical point of $K$, such that $K$ only depends on the signed distance $\rho$ from $\Sigma$ on $U$ 
		and 
		\begin{equation} \label{confsolitonmetric}
			g_1=d\rho^2+(K'(\rho))^2dt^2;
		\end{equation}
		furthermore, the tubular neighborhood of 
		$\Sigma$ can be chosen of
		the form $U=\{\rho_*<\rho<\rho^*\}$, where 
		$\rho_*\in [-\infty, 0)$ and $\rho^*\in (0,+\infty]$ \cite{catmantmazz}.
		We know that, for every complete Riemannian surface which is endowed with a conformal gradient soliton structure, 
		the Gaussian curvature $K$ must satisfy
		\[
		K(\rho)=-\dfrac{f'''(\rho)}{f'(\rho)}
		\]
		outside of the critical set of the potential function $f$.
		Since, in our case, 
		$f=K$, the Gaussian curvature $K$ must be a solution of
		the ODE
		\[
		K(\rho)\cdot K'(\rho)=-K'''(\rho)
		\]
		whenever $\nabla K\neq 0$, which can be
		seen as a \emph{static KdV equation}. 
		By integrating the ODE twice, we obtain that
		\begin{equation} \label{odederiv}
			(K')^2=P(K), \qquad 
			P(K):=-\dfrac{1}{3}K^3+AK+B,
		\end{equation} 
		where $A,B\in\RR$: 
		this means that we can express the solution implicitly, outside critical points, as 
		\begin{equation} \label{implicsol}
			\int_{K(0)}^{K(\rho)}\dfrac{dK}{\sqrt{P(K)}}=\pm\rho. 
		\end{equation}
		Since $\#\operatorname{crit}(K)\leq 1$, 
		around a regular level set $\Sigma$, without loss of generality we can assume that $\rho^*=+\infty$ (if $K$ has no critical points, we can simultaneously choose $\rho_*=-\infty$). By smoothness of $g_1$, $K$
		cannot blow up for finite distances, so we suppose that $K\to \infty$ as $\rho\to +\infty$.
		However, we immediately observe that 
		$K$ cannot tend to $+\infty$, since it would contradict \eqref{odederiv}. Assuming that $K\to -\infty$ as $\rho\to +\infty$ gives us a contradiction as well: indeed, 
		when $K\to -\infty$, we have that $\sqrt{P(K)}\approx (-K)^{3/2}$, which means that the integral in \eqref{implicsol} is finite, while $\rho\to +\infty$. Since $U$ does not contain the possible critical point of $K$, this implies that 
		$K$ must be bounded as $\rho\to +\infty$: if $\#\operatorname{crit}(K)=0$, the same conclusion
		holds when $\rho\to -\infty$.    
		\begin{itemize}
			\item Assume that $P(K)$ has three, distinct
			real roots, i.e.
			\[
			P(K)=-\dfrac{1}{3}(K-a)(K-b)(K-c),
			\]
			where $a,b,c\in\RR$, $a<b<c$: by \eqref{odederiv}, we immediately obtain that 
			either $K\leq a$ or $b\leq K\leq c$. In the first case, by monotonicity and boundedness of $K$ we would get that $K\to d\in\RR$ as 
			$\rho\to +\infty$: however, since $K$ is smooth, $K'$ approaches $0$ as $\rho\to+\infty$, which, by the expression of $P(K)$ and the fact that $K\leq a$, means that $d=a$. Since $\sqrt{P(K)}\approx
			\sqrt{a-K}$ as $K\to a$, the integral in \eqref{implicsol} (computed between $K(0)$ and $a$) must be finite, while $\rho\to+\infty$, which is impossible. 
			
			This means that $b\leq K\leq c$: if we assume that there exist $-\infty<\rho_b,\rho_c<+\infty$ such that $K\to b$ as $\rho\to\rho_b$ and 
			$K\to c$ as $\rho\to\rho_c$, by \eqref{odederiv} we obtain that 
			$K'$ goes to $0$ as $\rho$ goes to 
			$\rho_b$ or $\rho_c$, which contradicts the fact that $\#\operatorname{crit}(K)\leq 1$. Without loss of generality, we may assume that $K\to b$ as $\rho\to+\infty$: following the same line of reasoning as before, we obtain a contradiction, which means that, if 
			$P(K)$ has three real distinct roots, the ODE has no non-constant solution. 
			\item Now, let 
			\[
			P(K)=-\dfrac{1}{3}(K-a)^3, \quad a\in\RR;
			\]
			by \eqref{odederiv}, we know that 
			$K\leq a$ and we can solve the ODE explicitly,
			finding that $a=0$ and 
			\[
			K(\rho)=-\dfrac{12}{(\rho-C)^2}, \quad
			C=\pm\sqrt{\dfrac{12}{-K(0)}}.
			\]
			If we assume that $\rho_*>-\infty$, we would get that 
			$K\to 0$ as $\rho\to\rho_*$, but this is impossible, which means that $\rho_*=-\infty$: however, this gives a contradiction, since $K$ blows up when $\rho\to C$, 
			and this shows that there
			are no non-constant solutions.
			\item Now, assume that 
			\[
			P(K)=-\dfrac{1}{3}Q(K)(K-a),
			\] 
			where $a\in\RR$ and $Q(K)=K^2+a_1K+a_2$ is an irreducible
			polynomial of degree $2$ in $K$. Again, 
			we have $K\leq a$ and, as in the case where $P(K)$ has three distinct real roots, we obtain
			that $K\to a$ as $\rho\to +\infty$, showing the same contradiction. 
			\item Finally, assume that 
			$\#\operatorname{crit}(K)=0$ and that 
			\[
			P(K)=-\dfrac{1}{3}(K-a)^2(K-b), \quad
			a,b\in\RR.
			\]
			If $a<b$, since $K$ has no critical points 
			we must have $a<K<b$ or $K<a$: in the first case, 
			since we can choose $\rho_*=-\infty$ and 
			$K$ cannot go to $a$ or $b$ in finite distances (otherwise, we would have critical values for $K$ by \eqref{odederiv}), 
			by monotonicity of $K$ we can assume that $K\to b$ as $\rho\to +\infty$, which is impossible, as it can be seen applying the same argument as before using \eqref{implicsol}. In the second case, since $K$ must be bounded above and below by the previous reasoning, 
			we would get that $K\to a$ as $\rho\to\pm\infty$, which contradicts the fact that $K$ is monotone;
			on the other hand, if $a>b$, we would simply have that $K<b$: as before, $K\to b$ as $\rho\to \infty$, which is again impossible. 
		\end{itemize} 
		The only case we did not cover is when 
		$\#\operatorname{crit}(K)=1$ and
		\[
		P(K)=-\dfrac{1}{3}(K-a)^2(K-b):
		\]
		we want to show that, with this configuration,
		there exists a unique non-constant solution $K$. 
		First, assume that $a$ is the critical value for 
		$K$: then, since $K<b$, we must have $a<b$. 
		Furthermore, assuming that $\rho^*=+\infty$, since $K$ has a critical point we must have that 
		$K\to a$ as $\rho\to \rho_*$, i.e. the critical value must be achieved for a finite distance. However, this is impossible, since the integral appearing in \eqref{implicsol}, computed between $K(0)$ and $a$, is not finite in this case: this means that the critical value for $K$ must be $b$.
		Furthermore, a simple computation shows that
		\begin{align*}
			-\dfrac{1}{3}K^3+AK+B=
			-\dfrac{1}{3}K^3+\dfrac{1}{3}K^2
			(b+2a)-\dfrac{1}{3}aK(a-2b)+\dfrac{1}{3}a^2b,
		\end{align*}
		which holds if and only if
		\[
		\begin{cases}
			b&=-2a\\
			A&=a^2\\
			B&=-\dfrac{2}{3}a^3
		\end{cases};
		\]
		in particular, 
		\[
		P(K)=-\dfrac{1}{3}(K-a)^2(K+2a).
		\]
		
		From now on, we rely on the well-known work of Ishihara and Tashiro on conformal gradient solitons \cite{ishiharatashiro, tashiro}, as well as on
		results contained in \cite{osgood}: for the sake of completeness, we include here the full argument, even though some facts will be a direct consequence of their work. Since $\#\operatorname{crit}(K)=1$, we know that
		$(M,g)$ is conformally equivalent to the Euclidean plane $\RR^2$ or to the hyperbolic plane $\mathbb{H}^2$ (we point out that Tashiro forgets to include the Euclidean space case in his statement of Theorem 1 \cite{tashiro}). Let $p_0$ be the unique critical point for $K$ and let 
		$B$ be a small geodesic ball around $p_0$ such that we can express $g_1$ in polar geodesic coordinates, i.e. 
		\begin{equation} \label{polargeodcoord}
			g_1=dr^2+(h(r,\theta))^2d\theta^2
		\end{equation}
		for some smooth function $h$.
		By the work of Tashiro \cite{tashiro}, we know that, up to reparameterization, the geodesics emanating from $p_0$ are $K$\emph{-curves}, i.e. they are integral curves for $\nabla K$, and all the distance circles around the critical point are actually regular level sets of $K$: since $K$ only depends on the signed distance from these level sets, we also obtain that $K$ is radial on $B$, i.e. $K=K(r)$. This also follows directly by observing that 
		the vanishing of $\mathring{\hess}(K)$ implies the validity of the ODE
		\[
		\begin{cases}
			&\nabla_{\dot{\gamma}(r)}W(r)=\dfrac{\Delta K} {2}(\gamma(r))\dot{\gamma}(r)\\
			&W(0)=0,
		\end{cases}
		\]
		for every geodesic $\gamma(r)=\exp_{p_0}(rv)$ in $B$, where $v$ is a unit tangent vector at $p_0$ and $W(r):=\nabla K(\gamma(r))$;
		this, in turn, gives
		\[
		\nabla K(\gamma(r))=a(r)\cdot\dot{\gamma}(r),
		\]
		where $\dot{\gamma}(r)=\partial_r$ and 
		\[
		a(r)=\int_0^r\dfrac{\Delta K}{2}(\gamma(s))ds. 
		\]
		Now, we can use again the conformal
		soliton equation $\hess(K)=(\Delta K/2)g_1$,
		together with \eqref{polargeodcoord}, in order to
		obtain that, outside $p_0$,
		\[
		\dfrac{K''(r)}{K'(r)}=\dfrac{\partial_r h}{h}
		\Longrightarrow 
		h(r,\theta)=w(\theta)\cdot K'(r),
		\]
		for some smooth function $w$. We point out that, since $g_1$ is smooth on $M$, the expression \eqref{polargeodcoord} can be extended to the critical point and this implies that 
		\[
		h(0,\theta)=0, \qquad \partial_r h(0,\theta)=1:
		\]
		in particular, this tells us that, for every $\theta$, $h(r,\theta)>0$ in $B$ and, moreover, 
		$h$ must be an odd function in $r$. 
		We can exploit the condition on $\partial_rh$ to compute 
		\[
		1=\lim_{r\to 0}\dfrac{h(r,\theta)}{r}=
		w(\theta)\lim_{r\to 0}\dfrac{K'(r)}{r}=
		w(\theta)K''(0);
		\]
		this shows that $K''(0)\neq 0$ and that $w$ is
		constant, i.e. $h$ only depends on $r$ and
		\begin{equation} \label{hintermsofKprime}
			h(r)=\dfrac{1}{K''(0)}K'(r).
		\end{equation}
		In particular, inside $B$, the metric $g_1$ can
		be written as 
		\begin{equation} \label{radialwarpedmetric}
			g_1=dr^2+(h(r))^2d\theta^2,
		\end{equation}
		where, using the smoothness of $g_1$ on $M$ again,
		\[
		h(0)=0, \quad h'(0)=1, \quad h \mbox{ is odd}.
		\]
		Now, we want to show that \eqref{radialwarpedmetric} holds globally on $M$. 
		In order to see this, first we recall that, up to reparameterization, geodesics from $p_0$ are integral curves
		of $\nabla K$ for all $r>0$, which means that
		they cannot intersect at any other point different from $p_0$: this
		means that the cut locus $\operatorname{Cut}(p)$ is non-empty
		if and only if there exist points conjugate to $p$. 
		Let $r_0=\operatorname{inj}(p)$ and consider 
		the open ball $B_{r_0}$ centered at $p$, where $g$ is expressed
		as in \eqref{radialwarpedmetric}. We assume that
		there exists $q\in\partial B_{r_0}$ conjugate to $p$
		(hence, $r_0<+\infty$),
		which means that there exists a smooth, non-trivial Jacobi field
		$J=J(r)$ along the geodesic $\gamma$ connecting
		$p$ to $q$ such that $J(0)=0$ and $\lim_{r\to r_0}J(r)=0$,
		but this is impossible: indeed, a simple computation shows that 
		\[
		J(r)=c\cdot\partial_\theta, \mbox{ }c\in \RR\setminus\{0\} \quad \Longra \quad \abs{J(r)}=\abs{h(r)},
		\]
		which yields
		\[
		0=\lim_{r\to r_0}\abs{J(r)}=\lim_{r\to r_0}\abs{h(r)}.
		\]
		However, by \eqref{hintermsofKprime} and the smoothness of $K$, we would obtain
		that $\nabla K(q)\equiv 0$, which is impossible, since 
		$K$ has only one critical point. Since $(N_1,g_1)$ is complete,
		we conclude that $r_0=+\infty$ and, therefore,
		that
		$\operatorname{Cut}(p_0)=\varnothing$, which implies 
		that \eqref{radialwarpedmetric} holds 
		globally on $M$: in particular, $(M,g)$ is
		\emph{rotationally symmetric}. 
		
		Now, recall that 
		\[
		P(K)=-\dfrac{1}{3}(K-a)^2(K+2a):
		\]
		as said before, $-2a$ is the critical value
		for $K$ and, moreover, $K\leq -2a$ by \eqref{odederiv}, which means that 
		$-2a$ is the maximum of $K$. Assume that $a>0$:
		since $K\leq -2a<a$, we write 
		\[
		K'=-\dfrac{1}{\sqrt{3}}(a-K)\sqrt{-K-2a}. 
		\]
		We point out that, in \eqref{odederiv}, 
		$K$ is a function of the signed distance $\rho$
		from a regular level set $\Sigma$: however, since
		$h(r)=(K'(r)/K''(0))$ and both $r$ and $\rho$ are arc length parameterizations of the geodesics from $p_0$, we can simply pass
		from \eqref{confsolitonmetric} to \eqref{radialwarpedmetric} with a change of coordinates and interpret $K$ in \eqref{odederiv} and \eqref{implicsol} as a function of $r$.
		We have that
		\[
		\dfrac{r}{\sqrt{3}}=-\int_{-2a}^{K(r)}\dfrac{dK}{(a-K)\sqrt{-K-2a}}:
		\]
		making the standard substitution 
		$u=\sqrt{-K-2a}$, it is immediate to obtain
		\[
		\arctan\pa{\sqrt{\dfrac{-K-2a}{3a}}}=\dfrac{r\sqrt{a}}{2} \quad \Longra \quad
		-\dfrac{\pi}{\sqrt{a}}<r<\dfrac{\pi}{\sqrt{a}},
		\]
		which is impossible, since $K$ cannot blow up
		for finite radii $r$. 
		
		If $a=0$, we get that $K(0)=0$ and
		\[
		K'=-\dfrac{1}{\sqrt{3}}(-K)^{\frac{3}{2}},
		\]
		i.e. 
		\[
		\dfrac{r}{\sqrt{3}}=-\int_{0}^{K(r)}\dfrac{dK}{(-K)^{\frac{3}{2}}},
		\]
		which is impossible, since the integral diverges:
		hence, the only possibility is that $a<0$ and,
		therefore, $a<K\leq -2a$. 
		We obtain
		\[
		\dfrac{r}{\sqrt{3}}=-\int_{-2a}^{K(r)}\dfrac{dK}{(K-a)\sqrt{-K-2a}};
		\]
		by standard computations, we can show that 
		\begin{equation} \label{nonconstsolution}
			K(r)=3a\tanh^2\pa{\dfrac{r\sqrt{-a}}{2}}-2a.
		\end{equation}
		We can also compute $h$, due to the fact that,
		for every $r>0$, \eqref{hintermsofKprime} holds:
		indeed, since 
		\begin{equation*}
			K'(r)=3(-a)^{\frac{3}{2}}\tanh\pa{
				\dfrac{r\sqrt{-a}}{2}}\operatorname{sech}^2\pa{\dfrac{r\sqrt{-a}}{2}}, \quad
			K''(0)=-\dfrac{3}{2}a^2,
		\end{equation*}
		we obtain
		\begin{equation} \label{solutionh}
			h(r)=\dfrac{2}{\sqrt{-a}}\tanh\pa{
				\dfrac{r\sqrt{-a}}{2}}\operatorname{sech}^2\pa{\dfrac{r\sqrt{-a}}{2}}.
		\end{equation}
		We can define 
		\[
		H(r):=\exp\pa{\int_{r_0}^r\dfrac{1}{h(s)}ds}
		\]
		for some $r_0>0$: we observe that, since $h(0)=0$ and $h'(0)=1$,
		$H(r)\to 0$ as $r\to 0$ and, therefore, $H$ is a smooth,
		increasing map from $[0,+\infty)$ to $[0,R)$, where
		\[
		R:=\lim_{r\to +\infty}H(r).
		\]
		We point out that $g_1$ can be expressed as
		\[
		g_1=\pa{\dfrac{h(r)}{H(r)}}^2\pa{dH^2+H^2d\theta^2},
		\]
		which means that $g_1$ is conformal to the Euclidean metric
		on the ball in $\RR^2$ of radius $R$;
		by the previous discussion, it is immediate to see that 
		$R=+\infty$ in our case, which means that 
		$(N_1,g_1)$ is conformal to the Euclidean plane $\RR^2$.
		
		Finally, we note that that $K\to a<0$ as 
		$r\to +\infty$: since $h\to 0$, by \eqref{solutionh} 
		this means that $(N_1,g_1)$ has a hyperbolic cusp-like end and finite area, since 
		\[
		\operatorname{Area}(N_1)=
		\int_MdV_{g_1}=\int_0^{+\infty}\int_0^{2\pi}h(r) \, d\theta \, dr=-\dfrac{4\pi}{a}.
		\]
		We can apply the same computations for $g_2$ in order to conclude that one of the following holds:
		\begin{itemize}
			\item $K_1$ and $K_2$ are constant and, as observed before, $K_1\neq\pm K_2$;
			\item either $K_1$ or $K_2$ are given by 
			\eqref{nonconstsolution}. 
		\end{itemize}
	\end{proof}
	\subsection{A new example of a strictly Bach-flat metric}
	Exploiting the computations in the proof of Theorem \ref{bachpar22} when both $K_1$ and $K_2$ are not constant, we are able to construct a non-trivial one-parameter family of Bach-flat metrics on the product of two surfaces: namely, Bach-flat metrics which are neither locally conformally flat (and hence they are not half conformally flat) nor conformally Einstein.  
	\begin{proof}[Proof of Theorem \ref{thm example}]
		By direct inspection, we have that $(M,g)$ is neither LCF nor Einstein since $K_i$ is not constant and $K_1\neq\pm K_2$
		(also, we recall that a product metric on a product of surfaces is half conformally flat if and only if LCF);
		moreover, by the proof of Theorem \ref{bachpar22}, we know that $K_1$ and $K_2$ are solutions to the conformal solitons equation
		\begin{align}\label{confsol}
			\hess_i(K_i)=\dfrac{\Delta_i K_i}{2}g_i.
		\end{align}
		Hence, by \eqref{bach22comp} and \eqref{confsol}, we only need to check that
		\begin{align}\label{bachcomp22fl}
			B_{ab}&=\dfrac{1}{6}\sq{
				\Delta_1 K_1\delta_{ab}-\Delta_2K_2\delta_{ab}+\pa{K_1^2-K_2^2}\delta_{ab}};\\
			B_{ij}&=\dfrac{1}{6}\sq{
				\Delta_2 K_2\delta_{ij}-\Delta_1K_1\delta_{ij}+\pa{K_2^2-K_1^2}\delta_{ij}},\notag
		\end{align}
		are identically zero to verify that $(M,g)$ is Bach-flat; to do so, we set
		\[
		L_i(r_i)\coloneqq \tanh\pa{\dfrac{r_i\sqrt{-a}}{2}},
		\]
		so that
		\[
		h_i=\frac{2}{\sqrt{-a}}L_i\pa{1-L_i^2} \quad \mbox{ and }
		\quad K_i=3aL_i^2-2a,
		\]
		where we have omitted the dependence on $r_i$ to simplify the notation.
		As a consequence, by \eqref{hintermsofKprime} we can easily verify that 
		\[
		K_i'=-\frac{3}{2}a^2h_i\quad\text{and}\quad K_i''=-\frac{3}{2}a^2 h_i',
		\]
		where 
		\[
		h'_i=\frac{2}{\sqrt{-a}}\sq{L_i'(1-L_i^2)-2L_i^2L_i'};
		\]
		thus $\Delta_i K_i=-3a^2 h_i'=-3a^2(1-L_i^2)(1-3L_i^2)$, from which, substituting this expression into \eqref{bachcomp22fl}, we deduce 
		\begin{align}\label{compbaczero}
			B_{ab}&=\dfrac{1}{6}\sq{-3a^2(1-L^2_1)(1-3L^2_1)+3a^2(1-L^2_2)(1-3L^2_2)+(3aL_1^2-2a)^2-(3aL_2^2-2a)^2}\delta_{ab}=0;\\
			B_{ij}&=\dfrac{1}{6}\sq{-3a^2(1-L^2_2)(1-3L^2_2)+3a^2(1-L^2_1)(1-3L^2_1)+(3aL_2^2-2a)^2-(3aL_1^2-2a)^2}\delta_{ij}=0.\notag
		\end{align}
		To prove that $(M,g)$ is not conformal to a smooth Einstein metric we argue by contrdiction: let us consider $\tilde{g}=\varphi^{-2}g$, $0<\varphi \in\cinf$, such that 
		$\tilde{\ricc}=\lambda \tilde{g}$, $\lambda\in \RR$; then the components of $\tilde{\ricc}$ can be rewritten as
		\begin{align*}    \varphi^{-2}\tilde{R}_{\alpha\beta}=R_{\alpha\beta}+2\frac{\varphi_{\alpha\beta}}{\varphi}+\pa{\frac{\Delta_M\varphi}{\varphi}-3\frac{\abs{\nabla^M\varphi}_g}{\varphi^2}}\delta_{\alpha\beta},
		\end{align*}
		where $\alpha,\beta=1,2,3,4$, i.e., $(M,g)$ is conformally Einstein if there exists a positive smooth solution $\varphi$ of         
		\begin{align}\label{riccitilde} 
			\varphi\ricc+2\hess_M(\varphi)=\varphi \Lambda g,
		\end{align}        
		where $\Lambda\in\cinf$.
		Since $(M,g)$ is the Riemannian product of two surfaces, we have that $\ricc=K_1 g_1+K_2g_2$ and, hence, by \eqref{riccitilde} the mixed components $\varphi_{ai}$ of $\hess(\varphi)$ vanish identically, implying that
		\[
		\varphi=f_1+f_2,
		\]
		where $f_1\in C^{\infty}(N_1)$ and $f_2\in C^{\infty}(N_2)$. Moreover, by \eqref{riccitilde}, we deduce that
		\[
		(f_1)_{ab}=\varphi_{ab}=\frac{\varphi}{2}\pa{\Lambda-K_1}\delta_{ab},\quad (f_2)_{ij}=\varphi_{ij}=\frac{\varphi}{2}\pa{\Lambda-K_2}\delta_{ij},
		\]
		and, in particular,
		\begin{align}\label{deriv fi}
			0&=(f_1)_{abi}=\frac{1}{2}\pa{(\Lambda \varphi)_i-K_1(f_2)_i}\delta_{ab};\\
			0&=(f_2)_{ija}=\frac{1}{2}\pa{(\Lambda \varphi)_a-K_2(f_1)_a}\delta_{ij}.\notag
		\end{align}
		We now exploit \eqref{deriv fi} in order to obtain the equations 
		$0=(f_1)_{abia}$ and $0=(f_2)_{ijai}$ (not summing over $a$ and $i$):
		taking their difference, we obtain the identity
		\begin{equation} \label{productidentity}
			(K_1)_a(f_2)_i=(K_2)_i(f_1)_a.
		\end{equation}
		As a consequence, at each regular point $p\in N_2$ of $K_2$, we have that
		\[
		(f_1)_a=\frac{(f_2)_i(p)}{(K_2)_1(p)}(K_1)_a,
		\]
		i.e. there exists $C:=(f_2)_i(p)/(K_2)_i(p)\in\RR\setminus\{0\}$ such that $(f_1)_a=C\cdot (K_1)_a$ on $N_1$.
		Following the same line of reasoning, we can
		exploit \eqref{productidentity} again on a regular
		point of $K_1$ to obtain that there exists
		$C'\in\RR\setminus\{0\}$ such that 
		$(f_2)_i=C'\cdot (K_2)_i$: a simple inspection reveals
		that, in fact, $C'=C$ and,
		since $f_i$ is defined on $N_i$, we deduce that
		\[
		f_i=C K_i+D_i, \quad D_i\in\RR.
		\]
		It follows that the conformal factor $\varphi$ can be rewritten as
		\[
		\varphi=C(K_1+K_2)+D,
		\]
		where $D=D_1+D_2$
		and, using this identity, we get
		\begin{align}\label{phiric}
			\varphi\ricc=(C(K_1+K_2)+D)(K_1g_1+K_2g_2)=CK_1^2g_1+CK_2^2g_2+CK_1K_2g+DK_1g_1+CK_2g_2.  
		\end{align}
		On the other hand, exploiting the computations in \eqref{compbaczero}, we obtain that $\Delta_iK_i=a^2-K_i^2$ and, since $K_i$ is a solution of \eqref{confsol}, we have
		\begin{align}\label{hessphi}
			\hess_M(\varphi)=\hess_M(C(K_1+K_2)+D)=-\frac{C}{2}\pa{(K_1^2-a^2)g_1+(K_2^2-a^2)g_2}.
		\end{align}
		In particular, using \eqref{phiric} and \eqref{hessphi} into \eqref{riccitilde}, we conclude that $D=0$ and hence
		\[\varphi=C\pa{K_1+K_2};\]
		however, this is impossible, since
		\[
		Z:=\{(p,q)\in N_1\times N_2: K_1(p)+K_2(q)=0\}\neq\varnothing,
		\]
		and the conformal metric $\tilde{g}$ cannot be smoothly extended to $Z$, which concludes the proof.
	\end{proof}
	\begin{rem}
		If we consider a Riemannian four-manifold of the form $(N_1\times N_2, g=g_1\oplus g_2)$ and we assume that 
		$S_g\neq 0$ at every point, a result due to Derdzi\'{n}ski
		\cite[Lemma 18.4]{derdzinski}
		assures that $g$ is conformally Einstein if and only if 
		both factors satisfy \eqref{confsol}: furthermore, the 
		Einstein metric in the conformal class is exactly $\tilde{g}=\varphi^{-2}g$, where $\varphi=K_1+K_2$, which is also unique up to constant rescaling. The Riemannian metrics described in Theorem \ref{thm example} are, therefore, conformally Einstein on 
		$M\setminus Z$. 
	\end{rem}
	\subsection{Conformal gradient solitons, HCMU and extremal Hermitian metrics} \label{secction extHerm}
	Let $(M,g)$ be a complete Riemannian surface carrying 
	a conformal gradient soliton structure with potential function given by the Gaussian curvature $K$, i.e.
	\begin{equation} \label{HCMU}
		T:=\mathring{\hess}(K)=\hess(K)-\dfrac{\Delta K}{2}g\equiv 0.
	\end{equation}
	If $K$ is not constant, these metrics are also
	called \emph{HCMU metrics} and they have been extensively studied over the years, especially in the context of closed singular surfaces (see e.g. \cite{chenwuHCMU, chenwuxu, chenextremal,  chenobstruction}): for instance, the metric $g_i$ appearing in the statement of Theorem \ref{bachpar22} is one of the \emph{monopole metrics} appearing in \cite[Section 8.2]{chenextremal}. There exists a variational characterization of HCMU metrics, for which we first assume that $M$ is closed: it is well-known that, if $(M,g)$ is K\"{a}hler with respect to the complex structure $J$, \eqref{HCMU} is equivalent to the fact that $\nabla^{1,0} K$ is a holomorphic vector field.
	In turn, this holds if and only if $g$ is a critical point
	of the so-called \emph{Calabi functional} \cite{calabi1, calabi2}
	\[
	\mathfrak{E}(g):=\int_MK_g^2dV_g,
	\]
	where variations are taken inside the K\"{a}hler class
	of the K\"{a}hler form $\omega$ of $(M,g,J)$:
	these metrics are called \emph{extremal K\"{a}hler} and they can be seen as the ``best choices" of K\"{a}hler metrics inside a K\"{a}hler class (see e.g. \cite{szeke1, szeke2} for detailed discussions). 
	Although the Calabi functional can be defined on every compact complex manifold (replacing $K$ with the scalar curvature), in complex dimension 1 the extremal 
	K\"{a}hler condition can be translated as follows: $g$ is a critical point 
	of the Calabi functional with respect to conformal variations
	which fix the total volume of $(M,g)$, 
	i.e. $\tilde{g}(t)=e^{2tu}g$, where 
	\[
	\int_M u \, dV_g=0,
	\]
	and this leads to the Euler--Lagrange equation 
	$\nabla\mathfrak{E}(g)=c\cdot g$, $c\in\RR$, which reads
	\begin{equation} \label{eulerlagrangecalabiconf}
		\Delta K+K^2=C, \qquad C\in\RR.
	\end{equation}
	It is immediate to see that \eqref{eulerlagrangecalabiconf} is always implied by \eqref{HCMU}: indeed,
	by choosing 
	a local orthonormal frame $\{e_1,e_2\}$ and using
	the standard commutation formula for third derivatives,
	if $i\neq j$, by \eqref{HCMU} we have
	\[
	\dfrac{(\Delta K)_i}{2}=K_{jji}=K_{jij}-K\cdot K_i=-\dfrac{1}{2}(K^2)_i \Longra 
	\nabla(\Delta K+K^2)=0,
	\]
	which is exactly \eqref{eulerlagrangecalabiconf} (of course we are not summing over $j$).
	On the other hand, since $M$ is closed, we can also see that the converse is true, i.e. \eqref{HCMU} and \eqref{eulerlagrangecalabiconf} are equivalent:
	indeed, if \eqref{eulerlagrangecalabiconf} holds, we have
	\[
	\diver(\operatorname{T})_i=K_{jij}-\dfrac{(\Delta K)_i}{2}=
	\dfrac{1}{2}(\Delta K + K^2)_i=0,
	\]
	which leads to
	\[
	0=\int_M\langle\diver(\operatorname{T}),\nabla K\rangle dV_g=
	-\int_M\langle\operatorname{T},\hess(K)\rangle dV_g=-\int_M\abs{\operatorname{T}}^2dV_g,
	\]
	i.e. \eqref{HCMU}. 
	
	However, if $(M,g)$ is complete and non-compact, it is not always true that \eqref{eulerlagrangecalabiconf} implies 
	\eqref{HCMU}, while the converse still holds: a Hermitian metric on $(M,J)$ that satisfies \eqref{eulerlagrangecalabiconf} is called \emph{extremal Hermitian}. In this case, we consider compactly supported, conformal variations that fix the relative volume, i.e. $\tilde{g}(t)=e^{2tu}g$, where $u$ has compact support and 
	\[
	\int_M u\, dV_g=0.
	\]
	These metrics have been studied for several years, especially in the context of surfaces admitting cusp or conical singularities at finitely many points (see e.g. 
	\cite{chenwuexplicit, chenwuxu, chenextremal, wangzhu}): again, the metric $g_i$ in the statement of 
	Theorem \ref{bachpar22} can be seen as an extremal Hermitian (in fact, HCMU) metric on $\SS^2\setminus\{p\}$, 
	with a single hyperbolic cusp singularity at infinity. Many
	(non-)existence and classification results concerning these
	metrics are, to the best of our knowledge, obtained by imposing finite energy and finite area conditions. 
	Here, we note that the strategy of the proof of Theorem \ref{bachpar22} allows us to obtain a classification result
	for HCMU metrics on complete surfaces, without any extra assumptions, and for extremal Hermitian metrics, only under a finite energy hypothesis on $\nabla K$.
	\begin{proof}[Proof of Corollary \ref{corextKahler}]
		The first statement is a direct consequence of the proof
		of Theorem \ref{bachpar22}. In order to prove the second claim, it is sufficient to show that the finite energy condition on $\nabla K$, together with \eqref{eulerlagrangecalabiconf}, implies \eqref{HCMU}. 
		By the previous computations, we know that 
		$\diver(\operatorname{T})\equiv 0$ on $M$: let 
		$B_R\subset M$ be an open ball of radius $R$ 
		and let us consider a cutoff function $\chi\in \cinf$
		such that 
		\[
		0\leq\chi\leq 1, \quad \chi\equiv 1 \mbox{ on } B_R, \quad 
		\chi\equiv 0 \mbox{ on } M\setminus B_{2R}, \quad
		\abs{\nabla\chi}\leq\dfrac{C}{R} \mbox{ on }
		B_{2R}\setminus B_{R}.
		\]
		We integrate by parts and use Cauchy--Schwarz in order to obtain 
		\begin{align*}
			0&=\int_M\chi^2\langle\diver(\operatorname{T}),\nabla K\rangle dV_g=
			-\int_M\sq{\chi^2\langle\operatorname{T},\hess(K)\rangle+
				2\chi\operatorname{T}(\nabla\chi,\nabla K)}dV_g\leq\\
			&\leq-\int_M\chi^2\abs{\operatorname{T}}^2dV_g+
			2\pa{\int_M\chi^2\abs{\operatorname{T}}^2dV_g}^{\frac{1}{2}}\pa{\int_M\abs{\nabla\chi}^2\cdot\abs{\nabla K}^2dV_g}^{\frac{1}{2}},
		\end{align*}
		i.e., since $\nabla K\in L^2(M)$,
		\[
		\int_{B_R}\abs{\operatorname{T}}^2dV_g\leq \int_M\chi^2\abs{\operatorname{T}}^2dV_g\leq
		4\int_M\abs{\nabla\chi}^2\cdot\abs{\nabla K}^2dV_g\leq \dfrac{C}{R^2}:  
		\]
		since, by the Monotone Convergence Theorem, the left-hand side goes to $\norm{\operatorname{T}}_{L^2(M)}^2$ as \linebreak
		$R\to +\infty$, we get that $\operatorname{T}\equiv 0$ on $M$ and this concludes the proof. 
	\end{proof}
	\begin{rem}
		It is worth to note that, if $M$ is closed, Corollary \ref{corextKahler} provides a proof of the well-known fact that extremal K\"{a}hler metrics on closed surfaces always have constant Gaussian curvature. 
	\end{rem}
	\section{Some properties of the global minimizers II: products of the type \texorpdfstring{$\SS^1\times N$}{S^1 X N}}
	\label{secction global min II}
	
	If we assume that $(M,g)$ isometrically splits as 
	a product of the form $\SS^1\times N$, where $(N,g_N)$ is a Riemannian three-manifold, so far we are not able to provide a full classification result for the Bach-parallel condition. Indeed, we can compute the local components of
	$\bach_g$ with respect to a local orthonormal coframe as
	\small
	\begin{align} \label{bach13comp}
		2B_{11}&=-\dfrac{1}{6}\Delta_N {}^NS-\dfrac{1}{2}\abs{{}^N\operatorname{E}}^2; \\
		2B_{ij}&=\Delta_N{}^NR_{ij}-\dfrac{1
		}{3}{}^NS_{ij}+\dfrac{7}{3}{}^NS\cdot{}^NR_{ij}
		-4{}^NR_{ij}^2+\pa{\dfrac{3}{2}\abs{{}^N\ricc}^2-\dfrac{5}{6}{}^NS^2-\dfrac{1}{6}\Delta_N{}^NS}\delta_{ij}:
		\notag
	\end{align}
	\normalsize
	if $\nabla\bach\equiv 0$, by Theorem \ref{splittingbach}
	we know that there exists $c\in\RR$ such that 
	$B_{ij}=c\delta_{ij}$ and $B_{11}=-3c$,
	which, by \eqref{bach13comp} and the fact that 
	$\abs{{}^N\operatorname{E}}^2=\abs{{}^N\ricc}^2-(1/3){}^NS^2$,
	implies
	\begin{equation} \label{criticaleqft}
		\Delta_N{}^NR_{ij}-\dfrac{1}{3}{}^NS_{ij}+
		\dfrac{7}{3}{}^NS\cdot {}^NR_{ij}-4{}^NR_{ij}^2+
		\pa{\dfrac{4}{3}\abs{{}^N\ricc}^2-\dfrac{7}{9}{}^NS^2-\dfrac{2}{9}\Delta_N{}^NS}\delta_{ij}=0.
	\end{equation}
	Now, a simple inspection of \eqref{criticaleqft} (see e.g. \cite{besse, catinocalcvar, catmastrmont}) shows the following
	(see Theorem \ref{criticalequivalenceintro} in the Introduction)
	\begin{theorem} \label{criticalequivalence}
		Let $(N,g_N)$ be a complete Riemannian three-manifold. Then, the product metric $g=g_{\SS^1}\oplus g_N$ defined on $M=\mathbb{S}^1\times N$ is Bach-parallel if and only if $g_N$ is critical for the Riemannian functional 
		\[
		\mathfrak{F}_{-1/3}(h)=\int_N\abs{{}^N\ricc_h}_h^2 \, dV_h -\dfrac{1}{3}\int_N{}^NS_h^2dV_h. 
		\]
	\end{theorem}
	The statement of Theorem \ref{criticalequivalence} needs
	some more detailed explanation; for the sake of simplicity, we will drop the superscript $N$ from the curvature components, unless there is some ambiguity. By Theorem \ref{splittingbach}, the Bach-parallel hypothesis for such a product metric can be rephrased as 
	\[
	\bach_g\restrict{\SS^1}=-3c\cdot g_{\SS^1}, \quad
	\bach_g\restrict{N}=c\cdot g_N
	\]
	for some $c\in\RR$: on the other hand, the gradient
	of the functional $\mathfrak{F}_{-1/3}$ on a three-dimensional manifold is given by 
	(see e.g. \cite{catinocalcvar})
	\[
	\nabla\mathfrak{F}_{-1/3}=-\Delta\ricc+\dfrac{1}{3}\hess(S)-\dfrac{7}{3}S\cdot\ricc+4\ricc\circ\ricc-\pa{
		\dfrac{3}{2}\abs{\ricc}^2-\dfrac{5}{6}S^2-\dfrac{1}{6}\Delta S}g_N,
	\]
	which is exactly the global version of the component 
	$-2B_{ij}$ in \eqref{bach13comp}. We can therefore say the following.
	\begin{itemize}
		\item If $\nabla\mathfrak{F}_{-1/3}=0$, we have that 
		$g_N$ is critical for $\mathfrak{F}_{-1/3}$ with
		respect to all variations, which are taken compactly supported if $N$ is complete non-compact. Therefore, 
		Theorem \ref{criticalequivalence} shows that 
		this is equivalent to $g$ being Bach-flat on $M$. If $N$ is closed, it is well known that this implies that 
		$g_N$ is an Einstein metric, since 
		\[
		0=\operatorname{tr}_{g_N}(\nabla\mathfrak{F}_{-1/3})=-\dfrac{1}{6}\Delta S-\dfrac{1}{2}\abs{\operatorname{E}}^2\Longra 
		0=\int_N\abs{\operatorname{E}}^2dV_{g_N};
		\]
		if $N$ is complete and non-compact, a result due to Anderson assures that $g_N$ is a flat metric, provided that $S\geq 0$ \cite{anderson2}.
		\item If $\nabla\mathfrak{F}_{-1/3}=c\cdot g_N$ for some Lagrange multiplier $c\in\RR\setminus\{0\}$, $g_N$ is critical for $\mathfrak{F}_{-1/3}$ with respect to variations that fix the total volume of $N$, in the closed case, or with respect to compactly supported variations that fix the relative volume, i.e. variations of the form 
		$\bar{g}_t=g_N+th$, where $h\in\operatorname{Sym}^2(T^*N)$ 
		such that 
		\[
		\int_{N}\operatorname{tr}(h)dV_{g_N}=0.
		\]
		In this case, Theorem \ref{criticalequivalence} tells that this is equivalent to $g$ being Bach-parallel and not Bach-flat. In the closed case, some rigidity results under curvature assumptions were obtained, for instance, in 
		\cite{bernardini, catinocalcvar, katagiri, tanno}; to the best of our knowledge, there are no results in the literature that concern variations with fixed relative volume in the complete, non-compact case.
	\end{itemize}
	In the closed case, if one 
	considers variations with fixed volume, it is possible to find non-Einstein critical metrics for $\mathfrak{F}_{-1/3}$:
	for instance, as a direct corollary of the classification
	result in \cite{garciario3homog}, by means of Theorem \ref{criticalequivalence} we can state the following
	\begin{proposition}
		Let $N$ be a simply connected Lie group of dimension $3$, endowed with a left-invariant metric $g_N$. Then the metric $g=g_{\SS^1}\oplus g_N$ on 
		$M=\SS^1\times N$ is Bach-parallel if and only if 
		\begin{enumerate}
			\item $(M,g)$ is locally conformally flat, if $N$ is not unimodular;
			\item $(M,g)$ is $\SS^1\times\SS^3$, $\SS^1\times\RR^3$ or $\SS^1\times SU(2)$, where $\SS^3$ and $\RR^3$ are endowed with their standard metric and $SU(2)$ is endowed with a Berger sphere metric, if $N$ is unimodular.
		\end{enumerate}
	\end{proposition}
	\noindent
	We recall that Berger sphere metrics are constructed by 
	rescaling the circle fibers of the Hopf fibration 
	$SU(2)\to \SS^2$ and they are, therefore, non Einstein in general (see e.g. \cite[Section 4.4.3]{petersen}); other 
	critical metrics for $\mathfrak{F}_{-1/3}$ on $SU(2)$ 
	were constructed, for instance, in \cite{sekigawa}.
	
	When we consider critical metrics for $\mathfrak{F}_{-1/3}$ on $N$, both for the unnormalized functional and for compactly supported variations preserving the relative volume, 
	we can provide a rigidity result (Theorem \ref{thmFintro}) under suitable assumptions on the volume growth of $N$, on the Ricci tensor of $g_N$ and its covariant derivative.
	\begin{proof}[Proof of Theorem \ref{thmFintro}]
		Since $g_N$ is critical for 
		$\mathfrak{F}_{-1/3}$, by Theorem \ref{criticalequivalence} and \eqref{bach13comp} we immediately get 
		\begin{equation} \label{pdescalar}
			\Delta S+3\abs{E}^2=18c
		\end{equation}
		and, by the assumption on $\abs{\ricc}^2$, we obtain
		\[
		\Delta S\geq S^2 \quad \mbox{ on } N;
		\]
		since $\abs{\ricc}$ is bounded on $N$, we also obtain that,
		in particular, $\ricc\geq -K\cdot g$ for some positive constant $K$, which implies that the Omori--Yau maximum principle for $\Delta$ holds on $N$. Since $S$ is bounded above by hypothesis, denoting by $S^*:=\sup_N S$ we know that there exists a sequence $\{p_j\}\subset N$ such that
		\[
		S(p_j)\geq S^*-\dfrac{1}{j}, \quad \Delta S(p_j)\leq \dfrac{1}{j}
		\]
		for every $j$: hence, we immediately obtain
		\[
		\dfrac{1}{j}\geq \Delta S(p_j)\geq S^2(p_j) \Longra 
		S^*=0,
		\]
		which means that $S\leq 0$ on $N$. 
		
		Now, let $u:=-S\geq 0$: by the previous discussion, 
		$u$ is a non-negative solution of 
		\[
		\Delta u+ u^2\leq 0.
		\]
		The volume growth hypothesis allows us to 
		use a result due to Grigor'yan and Sun \cite{grigsun} in order to conclude that $u=S\equiv 0$: furthermore, by
		\eqref{pdescalar}, we conclude that 
		$\abs{\ricc}^2=\abs{\operatorname{E}}^2$ is a constant
		function on $N$.
		Exploiting 
		\eqref{criticaleqft}, we get
		\begin{equation} \label{newcritequation}
			\Delta R_{ij}=4\pa{R_{ij}^2-\dfrac{1}{3}\abs{\ricc}^2\delta_{ij}}.
		\end{equation}
		At every $p\in N$, we choose a local orthonormal frame
		in an open neighborhood of $p$ which diagonalizes
		$\ricc$ at $p$ and we denote by
		\[
		\lambda_1\geq \lambda_2\geq\lambda_3
		\]
		the eigenvalues of $\ricc$: since $\ricc$ is trace-free, we have that $\lambda_1\geq 0$ and $\lambda_3\leq 0$ for every $p\in N$. First, we claim that 
		either $\lambda_1>0$ at every point or $g_N$ is flat:
		indeed, since $\abs{\ricc}^2$ is constant on $N$, 
		if $\lambda_1=0$ at some point $p\in N$, then
		$\lambda_2=\lambda_3=0$, since $\ricc$ is trace-free,
		which would imply that $\abs{\ricc}^2=0$ at $p$ and, hence,
		on $N$. Therefore, we assume that $g_N$ is not flat
		and, at every $p\in N$, we define the function
		\[
		\rho(p)\coloneqq\dfrac{\lambda_2(p)}{\lambda_1(p)}.
		\]
		Following the computations carried out in 
		\cite{catinodameno}, we can easily observe that 
		\[
		-\dfrac{1}{2}\leq \rho\leq 1 \quad \mbox{ on } N
		\]
		and that 
		\begin{align} \label{eigenvalueequations}
			\lambda_2&=\rho\lambda_1, \quad \lambda_3=-(1+\rho)\lambda_1,\\
			\abs{\ricc}^2&=2(1+\rho+\rho^2)\lambda_1^2, \notag
		\end{align}
		which imply
		\begin{equation} \label{determinantequation}
			\det(\ricc)=-\rho(1+\rho)\lambda_1^3=-\dfrac{\rho(1+\rho)}{\sq{2(1+\rho+\rho^2)}^{3/2}}\abs{\ricc}^3.
		\end{equation}
		We can contract \eqref{newcritequation} with
		$R_{ij}$ and in order to obtain 
		\[
		R_{ij}\Delta R_{ij}=4\operatorname{tr}(\ricc^3);
		\]
		since $\abs{\ricc}^2$ is constant and $\ricc$ is trace-free, by the standard Weitzenb\"{o}ck formula for tensors and \eqref{determinantequation}
		we get
		\begin{equation} \label{equalityriccnablaricc}
			\abs{\nabla\ricc}^2=-12\det(\ricc)=\dfrac{12\rho(1+\rho)}{\sq{2(1+\rho+\rho^2)}^{3/2}}\abs{\ricc}^3.
		\end{equation}
		As in \cite{catinodameno}, we observe that the function 
		\[
		F(\rho)\coloneqq\dfrac{\rho(1+\rho)}{\sq{2(1+\rho+\rho^2)}^{3/2}}
		\]
		is non-decreasing for $\rho\in [-1/2,1]$: our goal
		is to show that there exists $\alpha>0$ such that
		\[
		\rho\geq\alpha, \quad F(\rho)> 0 \quad \mbox{on } N.
		\] 
		This will allow us to show that $(N,g_N)$ has finite volume, which, adapting a technique appearing in \cite{catinodameno} (also derived from the work of LeBrun \cite{lebrun}) and using a capacity argument, will lead us to a contradiction, proving that $g_N$ must be flat, up to a double cover of $N$.
		\begin{itemize}
			\item[\textbf{Step 1.}] First, we prove that, since 
			$\abs{\ricc}\neq 0$ on $N$, there exists no open subset $U$
			of $N$ such that $\lambda_1=\lambda_2$ on $U$. 
			By contradiction, assume that such $U$ exists, 
			which implies 
			that $\lambda_3$ is simple on $U$ and that
			$\lambda_i\in C^{\infty}(U)$ for every $i=1,2,3$.
			Up to choosing $\tilde{U}\subset U$, we can assume that
			$U$ is a local chart of $N$: moreover, since the
			multiplicities of the eigenvalues are constant on $U$, 
			there exists a local orthonormal frame $\{e_1,e_2,e_3\}$
			with respect to which $\ricc$ is diagonal on $U$
			(we choose it in such a way that $e_1,e_2$ are eigenvectors
			for $\lambda_1$ and $e_3$ is an eigenvector for $\lambda_3$). Since $\ricc$ is trace-free, we obtain
			\[
			\lambda_1=\lambda_2=:\lambda, \quad \lambda_3=-2\lambda
			\]
			on $U$ and a simple computation shows that 
			\[
			\lambda=\dfrac{\abs{\ricc}}{\sqrt{6}},
			\]
			which implies that $\lambda$ is constant on $U$. By the
			standard formula for covariant derivatives of tensors,
			this yields $R_{33,k}=0$ for every $k$ on $U$: exploiting this fact, we can differentiate again and compute
			\[
			R_{33,kl}=-2R_{p3,k}\omega_3^p(e_l),
			\]
			where $\omega_j^i$ are the Levi--Civita connection forms
			associated to $\{e_1,e_2,e_3\}$. We observe that, since $\lambda$ is constant, if $p\neq 3$
			\begin{equation} \label{covderriccilambda}
				R_{p3,k}=e_k(R_{p3})-R_{q3}\omega_p^q(e_k)-R_{pq}\omega_3^q(e_k)=
				-3\lambda\omega_3^p(e_k),
			\end{equation}
			which implies that
			\[
			\Delta R_{33}=6\lambda\omega_3^p(e_l)\omega_3^p(e_l).
			\]
			Using \eqref{newcritequation} and the fact that
			$\abs{\ricc}^2=6\lambda^2$, we obtain
			\[
			\omega_3^p(e_l)\omega_3^p(e_l)=\dfrac{4}{3}\lambda;
			\]
			by recalling the definition of Levi--Civita connection forms, this is equivalent to
			\begin{equation} \label{normeigenvector}
				\abs{\nabla e_3}^2=\dfrac{4}{3}\lambda.
			\end{equation}
			Furthermore, since $S\equiv 0$, we know that $\ricc$ is
			divergence-free: a straightforward computation shows
			that
			\begin{align*}
				0&=R_{i1,i}=R_{31,3}=-3\lambda\omega_3^1(e_3) \Longra \omega_3^1(e_3)=0; \\
				0&=R_{i2,i}=R_{32,3}=-3\lambda\omega_3^2(e_3) \Longra \omega_3^2(e_3)=0; \\
				0&=R_{i3,i}=R_{13,1}+R_{23,2}=-3\lambda\sq{\omega_3^1(e_1)+\omega_3^2(e_2)}\Longra \omega_3^1(e_1)+\omega_3^2(e_2)=0;
			\end{align*}
			this immediately implies that
			\begin{equation} \label{vanishingdere3}
				\nabla_{3}e_3=\omega_3^k(e_3)e_k=0, \qquad 
				\diver e_3=\langle\nabla_{i}e_3,e_i\rangle=
				\omega_3^i(e_i)=0.
			\end{equation}
			Using \eqref{vanishingdere3}, it is easy to derive
			\begin{align*}
				0&=e_3(\diver(e_3))=\langle\nabla_{3}\nabla_{i}e_3,e_i\rangle+\langle\nabla_{i}e_3,\nabla_{3}e_i\rangle=\\
				&=-\langle\nabla_{\nabla_{i}e_3}e_3,e_i\rangle+
				\langle\nabla_{\nabla_{3}e_i}e_3,e_i\rangle-\langle R(e_i,e_3)e_3,e_i\rangle+\langle\nabla_{i}e_3,\nabla_{3}e_i\rangle=\\
				&=-\omega_3^k(e_i)\omega_3^i(e_k)+\omega_i^k(e_3)\omega_3^i(e_k)-R_{33}+\omega_3^k(e_i)\omega_i^k(e_3)=\\
				&=-\omega_3^k(e_i)\omega_3^i(e_k)-R_{33},
			\end{align*}
			i.e. 
			\[
			-2\lambda=R_{33}=-\omega_3^k(e_i)\omega_3^i(e_k).
			\]
			Now, we can see that 
			\begin{align*}
				\abs{\nabla e_3}^2&=\omega_3^k(e_i)\omega_3^k(e_i)=
				\dfrac{1}{2}\sq{\omega_3^k(e_i)\omega_3^k(e_i)+\omega_3^i(e_k)\omega_3^i(e_k)}=\\
				&=\dfrac{1}{2}\sq{\omega_3^k(e_i)-\omega_3^i(e_k)}\sq{\omega_3^k(e_i)-\omega_3^i(e_k)}+
				\omega_3^k(e_i)\omega_3^i(e_k)\geq 
				\omega_3^k(e_i)\omega_3^i(e_k),
			\end{align*}
			i.e., by \eqref{normeigenvector},
			\[
			\dfrac{4}{3}\lambda=\abs{\nabla e_3}^2\geq \omega_3^k(e_i)\omega_3^i(e_k)=2\lambda,
			\]
			which is impossible, since $\lambda\neq 0$. 
			
			\item[\textbf{Step 2.}] By the previous step, we know that
			there exists $p\in N$ such that $\lambda_1$ is simple at $p$ and, hence, in an open neighborhood $U$ of $p$. 
			We choose a local orthonormal frame $\{e_1,e_2,e_3\}$
			on $U$ such that $e_1$ is an eigenvector for $\lambda_1$
			in $U$ and that $\ricc$ is diagonal at $p$. 
			By differentiating $R_{11}=\lambda_1$, we obtain
			\[
			R_{11,k}=e_k(\lambda_1)
			\]
			and, taking another covariant derivative, we obtain
			\begin{align*}
				R_{11,kl}&=e_l(e_k(\lambda_1))-2R_{s1,k}\omega_1^s(e_l)-
				R_{11,s}\omega_k^s(e_l)=\\
				&=(\lambda_1)_{kl}-2R_{s1,k}\omega_1^s(e_l);
			\end{align*}
			we observe that, at $p$,
			\[
			R_{s1,k}=(\lambda_s-\lambda_1)\omega_s^1(e_k), 
			\]
			which means that, since $\lambda_1$ is the largest eigenvalue,
			\[
			\Delta R_{11}=\Delta\lambda_1+2\sum_{s\neq 1}(\lambda_s-\lambda_1)\omega_s^1(e_k)\omega_s^1(e_k)\leq \Delta\lambda_1
			\]
			(in the formula for $R_{s1,k}$ we are not summing over $s$).
			Combining this with \eqref{newcritequation}, we obtain
			\begin{equation} \label{diffineqeigen}
				\Delta\lambda_1\geq 4\pa{\lambda_1^2-\dfrac{1}{3}\abs{\ricc}^2}
			\end{equation}
			at every point where $\lambda_1$ is simple. 
			
			We claim that
			\[
			\lambda^*:=\sup_N\lambda_1\leq\dfrac{\abs{\ricc}}{\sqrt{3}}:
			\]
			this will give us the desired lower bound on $\rho$. 
			By contradiction, assume that $\lambda^*>\abs{\ricc}/\sqrt{3}$: by Step 1., we 
			know that, at every point $p'$ such that $\lambda_1(p')=\lambda_2(p')$, 
			\[
			\lambda_1(p')=\dfrac{\abs{\ricc}}{\sqrt{6}}<\lambda^*.
			\]
			We point out that, at $p'$, $\lambda_1$ is, in general, only a Lipschitz function: however, by the previous observation, there exists no sequence of points $\{p_j\}\subset N$
			such that $\lambda_1(p_j)$ has multiplicity two and 
			$\lambda_1(p_j)\to\lambda^*$ as $j\to +\infty$. 
			In order to overcome the lack of smoothness, we define 
			a smooth truncation of $\lambda_1$ as follows: let 
			$a_1,a_2\in\mathbb{R}$ such that 
			\[
			\dfrac{\abs{\ricc}}{\sqrt{6}}<a_1<\dfrac{\abs{\ricc}}{\sqrt{3}}<a_2<\lambda^*
			\]
			and let $\eta\in C^{\infty}(\RR)$ such that 
			\[
			\begin{cases}
				\eta(x)=0, \quad &\mbox{for every } x\leq a_1\\
				\eta(x)=x+C, \quad &\mbox{ for every } x\geq a_2\\
				\eta'\geq 0, \quad &\mbox{ on } \RR 
			\end{cases}
			\]
			for some $C\in\RR$. Now, we define the function $f$ on $N$ as
			\[
			f:=\eta(\lambda_1):
			\]
			it is clear that, by definition of $\eta$, $f\in C^{\infty}(N)$ and,
			since $\lambda_1$ is bounded above on $N$, so is $f$.
			Hence, we can apply the Omori--Yau maximum principle:
			there exists $\{p_j\}\subset N$ such that
			\[
			f(p_1)\to \sup_Nf, \quad \Delta f(p_j)\leq \dfrac{1}{j}.
			\]
			However, there exists $j_0$ such that, for every $j\geq j_0$, $\lambda_1(p_j)\geq a_2$, which means that
			$\Delta f(p_j)=\Delta\lambda_1(p_j)$: we stress the fact that, for every $j\geq j_0$, $\lambda_1>a_1$, which implies that, at $p_j$, $\lambda_1$ is a simple eigenvalue and, therefore, for every $p_j$, there exists an open neighborhood of $p_j$ on which $\lambda_1$ is smooth. Now, 
			by \eqref{diffineqeigen}, we conclude that
			\[
			\dfrac{1}{j}\geq \Delta\lambda_1(p_j)\geq 4
			\pa{\lambda_1^2(p_j)-\dfrac{\abs{\ricc}^2}{3}}\Longra 
			(\lambda^*)^2\leq \dfrac{\abs{\ricc}^2}{3},
			\]
			which is impossible and, therefore,
			the claim is proven. 
			
			By \eqref{eigenvalueequations}, we obtain
			\[
			\abs{\ricc}^2=2(1+\rho+\rho^2)\lambda_1\leq2(1+\rho+\rho^2)\cdot\dfrac{\abs{\ricc}^2}{3} \Longra 2\rho^2+2\rho-1\geq 0
			\Longra \rho\geq\dfrac{\sqrt{3}-1}{2};
			\]
			since we know that $F(\rho)$ is non-decreasing, 
			we get that 
			\[
			F(\rho)\geq F\pa{\dfrac{\sqrt{3}-1}{2}}=\dfrac{\sqrt{3}}{18};
			\]
			by \eqref{equalityriccnablaricc}, we deduce that 
			\begin{equation} \label{lowerboundricci}
				\abs{\nabla\ricc}^2\geq \dfrac{2}{\sqrt{3}}\abs{\ricc}^3.
			\end{equation}
			If we assume that $\mathrm{Vol}_{g_N}(N)=+\infty$,
			by the fact that $\abs{\nabla\ricc}\in L^2(N)$ we obtain
			\[
			+\infty>\int_N\abs{\nabla\ricc}^2dV_{g_N}\geq 
			\dfrac{2}{\sqrt{3}}\int_N\abs{\ricc}^3 dV_{g_N}=+\infty,
			\]
			which is impossible, since $g_N$ is not flat, which means that 
			$\mathrm{Vol}_{g_N}(N)$ is finite.
			\item[\textbf{Step 3.}] By the conclusion of Step 2, since $(N,g_N)$ has finite volume, we have that $(N,g_N)$ is also parabolic (see e.g. \cite[Corollary 1]{chengyau}). We want to use one particular characterization of parabolicity
			\cite[Theorem 5.1]{grigbook}: indeed, we know that $(N,g_N)$ is parabolic if and only if every compact subset 
			$K$ of $N$ has zero capacity, i.e. 
			\[
			\inf_{\phi\in \mathscr{L}(K)}\int_N\abs{\nabla\phi}^2dV_{g_N}=0, \qquad
			\mathscr{L}(K):=\set{\phi\in Lip_c(N): 0\leq\phi\leq 1, \mbox{ } \phi\restrict{K}\equiv 1}. 
			\]
			In particular, if we define an exhaustion by compact sets 
			of $N$, i.e. 
			\[
			N=\bigcup_jK_j, \quad K_j\subset K_{j+1},
			\]
			parabolicity implies that there exists a sequence 
			$\{\phi_j\}$ such that $\phi_j\in\mathscr{L}(K_j)$ and 
			\[
			\int_N\abs{\nabla\phi_j}^2dV_{g_N}\leq \dfrac{1}{j}
			\]
			for $j$ sufficiently large; in order to conclude, 
			we adapt arguments used in \cite{catinodameno} and inspired by
			methods developed in \cite{lebrun}. 
			
			First, arguing as in Step 1, we can immediately see that
			the smallest eigenvalue $\lambda_3$ of $\ricc$ must be negative everywhere, otherwise there would be a point where 
			$\ricc\equiv 0$: furthermore, by \eqref{equalityriccnablaricc},
			we have that $\det(\ricc)\leq 0$ on $N$, which implies that 
			$\lambda_3$ is a simple eigenvalue on $N$ and, hence, $\lambda_3\in C^{\infty}(N)$. This means that the eigenspace of $\lambda_3$ varies smoothly from point to point, i.e. the eigenspace determines a smooth, real line bundle $L$ on $N$. 
			At every $p\in N$, we have that there exist exactly two 
			unit eigenvectors $v$ and $-v$ associated to $\lambda_3$:
			therefore, up to passing to the double cover
			\[
			\tilde{N}:=\{w\in L: \abs{w}=1\}
			\]
			of $N$, we can assume that $L$ is a trivial line bundle and,
			hence, that there exists a global eigenvector $v$ for $\lambda_3$ on $\tilde{N}$ such that $\abs{v}=1$ with respect to the pullback metric $\pi^*(g_N)$ on $\tilde{N}$. It is clear that the double cover $\tilde{N}$ also has finite volume, which means that $(\tilde{N},\pi^*(g_N))$ is parabolic;
			from now on, for the sake of simplicity, we identify $\tilde{N}$ with $N$ and $\pi^*(g_N)$ with $g_N$. 
			
			We recall that, given $\alpha\in\Lambda^1N$, the 
			Bochner--Weitzenb\"{o}ck formula for 1-forms reads
			\begin{equation} \label{weitzenforms}
				\dfrac{1}{2}\Delta\abs{\alpha}^2=\abs{\nabla\alpha}^2-\langle\Delta_H\alpha,\alpha\rangle+\langle\ricc(\alpha),\alpha\rangle,
			\end{equation}
			where $\Delta_H\alpha=(d\delta+\delta d)\alpha$ is the Hodge Laplacian, $d$ is the exterior derivative and $\delta$ is its adjoint (see e.g. \cite{libook}); here, we are identifying the Ricci tensor with a linear, self-adjoint operator acting on $\Lambda^1N$, also denoted by $\ricc$. 
			Let $\eta$ be the unique $1$-form such that $\eta(w)=g_N(v,w)$, i.e. $\eta=v^{\flat}$: we know that $\eta$ is globally defined and, by the previous identification, we get
			\[
			\ricc(\eta)=\lambda_3\eta.
			\]
			From now on, whenever we want to carry out local computations, we choose a local orthonormal frame $\{e_1,e_2,e_3\}$ such that $e_3=v$: the dual coframe will then be $\{\omega^1,\omega^2,\omega^3=\eta\}$, with connection forms
			$\{\omega_j^i\}$.
			We integrate \eqref{weitzenforms} with $\alpha=\phi_j\eta$ to obtain
			\[
			0=\int_N\sq{\abs{\nabla(\phi_j\eta)}^2-\phi_j\langle\Delta_H(\phi_j\eta),\eta\rangle+\phi_j^2\lambda_3}dV_{g_N};
			\]
			we observe that, since $\eta=\omega^3$ and the $\omega_j^i$'s are skew-symmetric,
			\[
			\langle\nabla_k\eta,\eta\rangle=-\langle\omega_j^3(e_k)\omega^j,\omega^3\rangle=0,
			\]
			which means that, since $\abs{\eta}=1$,
			\[
			\abs{\nabla(\phi_j\eta)}^2=
			\abs{\nabla\phi_j}^2+\phi_j^2\abs{\nabla\eta}^2.
			\]
			Hence, by definition of $\Delta_H$, we obtain
			\[
			\int_N\abs{\nabla\phi_j}^2+\phi_j^2\lambda_3+\phi_j^2\abs{\nabla\eta}^2-\phi_j\pa{\abs{d\eta}^2+\abs{\delta\eta}^2}dV_{g_N}=0;
			\]
			recalling that $\abs{\nabla\eta}=\abs{\nabla v}$, we finally get
			\begin{equation} \label{fundamentalineq}
				\int_N\pa{-\phi_j^2\lambda_3-\phi_j^2\abs{\nabla v}^2}dV_{g_N}\leq\int_N\abs{\nabla\phi_j}^2dV_{g_N}.
			\end{equation}
			Now, we want to suitably estimate the second term in the left-hand side of \eqref{fundamentalineq}: by the computations
			performed in Step 2, we know that
			\[
			R_{i3,k}=(\lambda_3-\lambda_i)\omega_3^i(e_k)
			\]
			(we are not summing over $i$) and, hence, we can write
			\begin{align*}
				\abs{\nabla\ricc}^2&\geq 2\sum_k\sq{(R_{13,k})^2+(R_{23,k})^2}=\\
				&=2\sum_k[(\lambda_1-\lambda_3)^2(\omega_3^1(e_k))^2+
				(\lambda_2-\lambda_3)^2(\omega_3^1(e_k))^2].
			\end{align*}
			First, we note that, if $\lambda_2=0$ at $p$, 
			by \eqref{equalityriccnablaricc} we obtain $\abs{\nabla\ricc}^2=0$ at $p$, but this is impossible, 
			by the validity of \eqref{lowerboundricci}: hence, 
			we assume that $\lambda_2>0$ on $N$. 
			Since $\lambda_1\geq\lambda_2>0$ and $\lambda_3<0$, we
			immediately get that $(\lambda_1-\lambda_3)^2\geq (\lambda_2-\lambda_3)^2$, i.e.
			\[
			\abs{\nabla\ricc}^2\geq 2(\lambda_2-\lambda_3)^2\abs{\nabla v}^2,
			\]
			which comes from $\nabla_kv=\nabla_ke_3=\omega_3^i(e_k)e_i$. 
			Since $\lambda_2\neq\lambda_3$ on $N$ and $\ricc$ is trace-free, by \eqref{equalityriccnablaricc} we obtain
			\begin{equation*}
				\abs{\nabla v}^2\leq -\dfrac{6\lambda_1\lambda_2\lambda_3}{(\lambda_2-\lambda_3)^2}=-\dfrac{6\lambda_1\lambda_2\lambda_3}{(\lambda_1+2\lambda_2)^2};
			\end{equation*}
			now, we use the fact that $(\lambda_1+2\lambda_2)^2=
			(\lambda_1-2\lambda_2)^2+8\lambda_1\lambda_2$ to get
			the desired estimate
			\begin{equation} \label{estimate1}
				\abs{\nabla v}^2\leq -\dfrac{3}{4}\lambda_3.
			\end{equation}
			Finally, since $\lambda_3$ is a globally defined, smooth, negative function on $M$, inserting \eqref{estimate1} into \eqref{fundamentalineq} and exploiting the properties
			of the $\phi_j$'s, by the Monotone Convergence Theorem we obtain
			\[
			\dfrac{1}{j}\geq\int_N\abs{\nabla\phi_j}^2dV_{g_N}\geq 
			-\dfrac{1}{4}\int_N\phi_j^2\lambda_3\,dV_{g_N}\geq
			\dfrac{1}{4}\int_{K_j}-\lambda_3\, dV_{g_N}
			\underset{j\to+\infty}{\longra}\dfrac{1}{4}\int_N-\lambda_3\,dV_{g_N}>0,
			\]
			which is impossible and this concludes the proof. 
		\end{itemize}
	\end{proof}
	\begin{rem}\label{remvolgr}
		Let $(N,g_N)$ be a complete Riemannian manifold and consider the elliptic PDE 
		\[
		\Delta w+3\abs{\ricc}^2-w^2=c;
		\]
		as pointed out to us by Farina, if one drops the
		volume growth hypothesis we assumed to apply Grigor'yan--Sun's result, one might not be able to conclude that
		$w\equiv 0$. Indeed, if we consider $\mathbb{H}^3$ endowed
		with the standard hyperbolic metric
		(which has exponential volume growth), if $c=36$ and $u=-w$, we have
		\[
		-\Delta u=u^2
		\]
		on $\mathbb{H}^3$, since $\abs{\ricc}^2=12$: then, a result due to Mancini and Sandeep assures the existence
		of a positive, smooth solution $u$ 
		\cite[Theorem 1.4]{mancinisandeep}. Furthermore, 
		Grigor'yan and Sun construct counterexamples which show
		that the volume growth assumption is sharp, proving the
		existence of positive solution to the inequality
		\[
		\Delta u+u^\sigma\leq 0
		\]
		if one only imposes the weaker condition $\mathrm{Vol}(B_R)\leq C\cdot R^p\ln^qR$, 
		with either $p>2\sigma/(\sigma-1)$ or 
		$p=2\sigma/(\sigma-1)$ and $q>1/(\sigma-1)$
		(recall that, in our case, $\sigma=2$) \cite{grigsun}.
	\end{rem}
	\subsection{A rigidity result in the Cotton-flat case}
	Finally, we show that the vanishing of the Cotton tensor of a Bach-parallel, non-Bach-flat metric rules out the cases where $M\cong\mathbb{S}^1\times N$, with $N$ irreducible, and where $M$ is isometric to a product of surfaces such that one of them is endowed with a metric of the family described by Theorem \ref{bachpar22}. More precisely, we prove the following
	\begin{theorem} \label{bach3cott}
		Let $(M,g)$ be a complete Riemannian manifold 
		of dimension $n\geq 4$. Assume that $\nabla\bach\equiv 0$ on $M$: if 
		$\operatorname{C}\equiv 0$, then either $(M,g)$ is Bach-flat or its universal cover is isometric to a Riemannian product of harmonic curvature manifolds.
		In particular, if $n=4$, the following cases can occur:
		\begin{enumerate}
			\item $(M,g)$ is Bach-flat;
			\item $(M,g)$ is covered by a product of surfaces with constant curvatures $K_1$ and $K_2$ such that $K_1\neq\pm K_2$. 
		\end{enumerate}
	\end{theorem}
	\begin{proof}
		Up to passing to the universal cover, we can assume 
		that $M$ is simply connected. Suppose that 
		$\bach\not\equiv 0$: we want to 
		adapt a strategy due to 
		G\l{}odek \cite{glodek}.
		Since $\bach$ is parallel, we use again
		the commutation formulas for the second derivatives
		of $(0,2)$-tensors to obtain
		\[
		0=B_{ij,kl}-B_{ij,lk}=
		B_{pj}R_{pikl}+B_{ip}R_{pjkl}.
		\]
		Differentiating the previous equation, we have
		\[
		0=B_{pj}R_{klpi,t}+B_{ip}R_{klpj,t}:
		\]
		tracing with respect to $k$ and $t$, using
		the Ricci identity and the vanishing of $\operatorname{C}$ we obtain
		\begin{align*}
			0&=B_{pj}R_{klpi,k}+B_{ip}R_{klpj,k}=\\
			&=B_{pj}(R_{li,p}-R_{lp,i})+B_{ip}(R_{lj,p}-R_{lp,j})=\\
			&=\dfrac{1}{2(n-1)}\sq{B_{pj}(S_p\delta_{li}-S_i\delta_{lp})+B_{ip}(S_p\delta_{lj}-S_j\delta_{lp})},
		\end{align*}
		i.e.
		\begin{equation} \label{eqinterm}
			B_{jp}S_p\delta_{ik}-B_{jk}S_i+B_{ip}S_p\delta_{jk}-B_{ik}S_j=0.
		\end{equation}
		Since $\bach$ is trace-free, tracing with
		respect to $i$ and $k$ we get
		\begin{equation} \label{vanishscalbach}
			B_{jp}S_p=0.
		\end{equation}
		Inserting this into \eqref{eqinterm}, we 
		have
		\[
		B_{jk}S_i+B_{ik}S_j=0
		\]
		and, contracting with $S_j$ and using again
		\eqref{vanishscalbach}, we conclude that
		\[
		B_{ik}\abs{\nabla S}^2=0:
		\]
		in particular,
		\begin{equation} \label{bachscal}
			\abs{\bach}^2\cdot\abs{\nabla S}^2\equiv 0.
		\end{equation}
		Suppose that there exists $p\in M$ such that 
		$\abs{\nabla S}^2\neq 0$ at $p$. By continuity, 
		$\abs{\nabla S}^2\neq 0$ in an open neighborhood
		$U$ of $p$: by \eqref{bachscal}, this means that
		$\abs{\bach}^2\equiv 0$ on $U$. However, $\nabla\bach\equiv 0$ on $M$, which means that 
		$\abs{\bach}^2$ is a constant function on $M$: therefore, $\bach\equiv 0$ on $M$, which is a contradiction, hence $S$ is constant and, therefore, $(M,g)$ has harmonic curvature.  
		
		By Theorem \ref{splittingbach}, since 
		$\bach\not\equiv 0$ we have that $(M,g)$ is 
		isometric to a Riemannian product: without loss
		of generality, we can assume that
		\[
		M\cong (N_1^k\times N_2^{n-k}, g_1\oplus g_2).
		\]
		Since $(M,g)$ has harmonic curvature, we can also 
		immediately notice that $(N_i,g_i)$ has harmonic curvature. Furthermore, Theorem \ref{splittingbach}
		also implies that
		\[
		B_{ij}=(n-k)C\delta_{ij}, \quad 
		B_{ia}=0, \quad B_{ab}=-kC\delta_{ab},
		\]
		where $C\in\RR$, $i,j=1,2,...,k$ and 
		$a,b=k+1,...,n-k$. 
		
		Now, let $n=4$ and assume that $k=2$, where 
		$(N_1,g_1)$ and $(N_2,g_2)$ are 
		Riemannian surfaces. Since both manifolds have
		harmonic curvature, the scalar curvatures of 
		$g_1$ and $g_2$ must be constant: this immediately
		implies that the Gaussian curvatures $K_1$ and $K_2$ are constant, with $K_1\neq\pm K_2$, as
		observed in the proof of Theorem \ref{bachpar22}.

		Now we assume that $(M,g)$ is isometric to $(\mathbb{S}^1\times N\, , g_{\mathbb{S}^1}\oplus g_N)$, where $(N,g_N)$ is a three-dimensional, irreducible Riemannian manifold. By \eqref{bach13comp}, since ${}^NS$ is constant and $\nabla\bach\equiv 0$, we have that
		\[
		2B_{11}=-\dfrac{1}{2}\abs{{}^N\operatorname{E}}^2 
		\Longra {}^N\nabla\abs{{}^N\operatorname{E}}^2\equiv 0;
		\]
		furthermore, since $\operatorname{C}\equiv 0$, 
		both ${}^N\ricc$ and ${}^N\operatorname{E}$ are
		Codazzi tensors. 
		We use the fact that
		${}^NS$ and $\abs{{}^N\ricc}$ are constant and the Codazzi condition on ${}^N\ricc$ to obtain
		\begin{align*}
			(\Delta_N{}^NR_{ij})_k&=
			{}^NR_{ij,ttk}={}^NR_{it,jtk}=\pa{{}^NR_{pt}{}^NR_{pijt}+{}^NR_{ip}{}^N
				R_{ptjt}}_k=\\
			&=\sq{{}^NR_{pt}
				\pa{{}^NR_{pj}\delta_{it}-{}^NR_{pt}\delta_{ij}+{}^NR_{it}\delta_{pj}-{}^NR_{ij}\delta_{pt}-
					\dfrac{{}^NS}{2}(\delta_{pj}\delta_{it}-
					\delta_{pt}\delta_{ij})}+{}^NR_{ij}^2}_k=\\
			&=3{}^NR_{ij,k}^2-\dfrac{3}{2}{}^NS\cdot{}^NR_{ij,k};
		\end{align*}
		differentiating the second equation in \eqref{bach13comp} and using the previous identity, we get
		\[
		R_{ij,k}^2=\dfrac{5}{6}{}^NS\cdot{}^NR_{ij,k},
		\]
		which, by the fact that ${}^NS$ is constant, yields
		\[
		E_{ij,k}^2=-\dfrac{1}{6}{}^NS\cdot{}^NE_{ij,k}.
		\]
		Hence, we immediately have
		\[
		{}^NE_{ij,k}^2-{}^NE_{ik,j}^2=
		-\dfrac{1}{6}{}^NS({}^NE_{ij,k}-{}^NE_{ik,j})=0:
		\]
		hence, ${}^N\operatorname{E}$ is a trace-free, symmetric, Codazzi tensor such that 
		${}^N\operatorname{E}^2$ is Codazzi as well, which, by \cite[Theorem 1.3]{umehara}, implies
		that ${}^N\nabla{}^N\operatorname{E}\equiv 0$. Since we are assuming that $N$ is irreducible, by Theorem \ref{splittingbach} applied to $N$ and $\operatorname{T}=\operatorname{E}$, we obtain that $(N,g_N)$ is an Einstein manifold, but this is impossible: indeed, this would imply that $(M,g)$ is locally conformally flat and, hence, Bach-flat, which contradicts our hypothesis.
		Therefore, by Theorem \ref{splittingbach}, 
		\[
		(M,g)\cong(\RR^2\times P^2, g_{\RR^2}\oplus g_P),
		\]
		where $(P^2,g_P)$ is a simply connected Riemannian surface with
		non-vanishing Gaussian curvature and this
		concludes the proof. 
	\end{proof}
	\begin{rem} \label{remglodek}
		We point out that the first part of the proof 
		of Theorem \ref{bach3cott} is purely algebraic: hence, by adapting the same argument, we can conclude that, given a complete Riemannian manifold $(M,g)$ of dimension at least $4$ with vanishing Cotton tensor and a trace-free, symmetric $(0,2)$-tensor $\operatorname{T}$ on $M$, if $\operatorname{T}$ is parallel then either $\operatorname{T}\equiv 0$ or the manifold splits as a product of harmonic curvature manifolds.
	\end{rem}
	\bibliographystyle{abbrv}
	\bibliography{Nabla_Bach_biblio}
\end{document}